\documentclass{amsart}
\usepackage{amssymb}
\usepackage{amsmath,amsfonts,amsthm,amstext}
\usepackage[all]{xy}
\usepackage[dvipsnames]{color}
\usepackage{stmaryrd}
\usepackage[normalem]{ulem} 
\usepackage{cancel} 
\usepackage{graphicx,xypic}
\usepackage[latin1]{inputenc}

\newtheorem{theorem}{Theorem}[section]
\newtheorem{lemma}[theorem]{Lemma}
\newtheorem{definition}[theorem]{Definition}
\newtheorem{prop}[theorem]{Proposition}
\newtheorem{cor}[theorem]{Corollary}
\newtheorem{claim}[theorem]{Claim}
\newtheorem{rmk}[theorem]{Remark}
\def\mapnew#1{\smash{\mathop{\longrightarrow}\limits^{#1}}}

\def\-{\overline}

\newcommand{\Z}{\mathbb{Z}}
\newcommand{\V}{\mathcal{V}}
\newcommand{\K}{\hspace{2pt} \mathbb{K}}

\newcommand{\Ker}{\operatorname{Ker}} 

\newcommand{\ld}{\lambda_\Delta}
\newcommand{\ee}{\varepsilon}
\newcommand{\ract}{\triangleleft}

\newcommand{\pt}{\tilde \pi}
\newcommand{\kpg}{\K_{par} G}
\newcommand{\one}{1_{\K\Delta}}
\newcommand{\idB}{1}

\newcommand{\notni}{  \hspace{1pt } / \hspace{-5pt }\ni}

\numberwithin{equation}{section}

\def\Vightarrow#1{\smash{\mathop{\longrightarrow}\limits^{#1}}}

\newcommand{\kparg}{\K_{par}G}
\newcommand{\kd}{\K \Delta}

\newcommand{\m}{{}^{-1}}

\newcommand{\A}{{\mathcal A}}

\usepackage[dvipsnames]{xcolor}

\begin{document}

\title{Homology and cohomology via the partial group algebra}

\author{Marcelo Muniz Alves, Mikhailo Dokuchaev, Dessislava H. Kochloukova}

\address
{Federal University of Paran\'a (UFPR), Brazil; University of S\~ao Paulo (USP), Brazil; State University of Campinas (UNICAMP), Brazil} 
\email{marcelomsa@ufpr.br, dokucha@gmail.com, desi@unicamp.br}


\date{}


\begin{abstract}  We study partial homology and cohomology from ring theoretic point of view via the  partial group algebra $\kpg$. In particular, we link the partial homology and cohomology of a group $G$ with coefficients in an irreducible (resp. indecomposable) $\kpg$-module $M$ with the ordinary homology and cohomology groups of a subgroup $H$ of $G,$ where $H$ depends on $M,$ with coefficients in an appropriate irreducible (resp. indecomposable) $\K H$-module. Furthermore, we compare the standard cohomological dimension $cd_{\K}(G)$ (over a field $\K$) with the partial cohomological dimension $cd_{\K}^{par}(G)$ (over $\K$) and show that  $cd_{\K}^{par}(G) \geq cd_{\K}(G)$  and that there is equality for $G = \Z$.

\end{abstract}

\maketitle

\section{Introduction}

 Partial actions and partial representations enjoy diverse  theoretic developments and keep finding notable  applications.  They have been brought into  use in the theory of $C^*$-algebras, providing, in particular, an  efficient approach to algebras generated by partial isometries, the Cuntz-Krieger algebras and the more general  Exel-Laca  algebras being  prominent early examples of the success of the new notions \cite{E2}, \cite{EL}.   In addition, they    allowed  for an isomorphic identification of  any  second countable $C^*$-algebraic bundle, satisfying a mild regularity condition,
with the semidirect product bundle associated to a continuous twisted partial group action  on the unit fibre algebra \cite{E0}.  Apart from their relevance in the study of   K-theory, ideal structure and representations of the $C^*$-algebras under consideration, the new tools turned out to be also productive  in  dealing with the KMS equilibrium states of $C^*$-dynamical systems  \cite{EL2}.

The $C^*$-algebraic advances, the Exel's definition  of a partial action of a group on an abstract set \cite{E1}, the concept of a partial representation introduced in \cite{E1} and \cite{QR}, and the relation of  these notions to  the inverse semigroup $S(G)$  defined in \cite{E1} stimulated the  algebraic investigations on the subject,  initially made  in  \cite{DEP}, \cite{KL1},   \cite{St1}, \cite{St2}.
In algebra the new concepts are   useful   to graded algebras,   Hecke algebras,   Leavitt path algebras,  
  inverse semigroups,  restriction  semigroups  and automata (see the survey article \cite{D3} and the references therein).  
  Moreover, the partial action definition of the Thompson's groups $V$ and $\mathcal G _{2,1}$
provides an appropriate framework to deal with algorithmic problems for them \cite{Birget}. Some important decision problems, properly placed in the language  of inverse semigroups \cite{St5}, are related to    
extension  properties of partial bijections, to some of which partial actions are also handy    \cite{Coul2}. In particular, they participate in an interaction between model theory, profinite topology of groups and formal languages, being relevant  this way  to  the Ribes-Zalesskii properties of groups and other topologically defined conditions on groups.

 Amongst the recent applications, we mention the significance of partial actions and  partial representations to  $C^*$-algebras related to dynamical systems of type $(m,n)$ \cite{AraEKa},  to dynamical systems associated to separated graphs and related $C^*$-algebras   \cite{AraE1}, \cite{AraL},  to paradoxical decompositions \cite{AraE1}, to shifts  \cite{AraL}, to full or reduced $C^*$-algebras  of $E$-unitary or strongly $E^*$-unitary inverse semigroups  \cite{MiSt}, 
to  topological higher rank graphs \cite{RenWil}, to Matsumoto and Carlsen-Matsumoto $C^*$-algebras of arbitrary subshifts    \cite{DE2} (see also  \cite{ESt1}, \cite{ESt2}),   to ultragraph $C^*$-algebras \cite{GR3} and to expansions of monoids in the class of two-sided restriction monoids  \cite{Kud2}, the latter involving a previous construction  from \cite{Kud} based on partial actions (see also \cite{CornGould}).

  Besides being related to inverse semigroups,  partial actions and partial representations also have various connections to groupoids \cite{D3}; one of them is based on the partial action groupoid, another one
 involves    the partial group algebra, which    is crucial for the present article, and  one more is the fact  proved in \cite{BeuGon2} and \cite{HazratLi}, stating that Steinberg algebras (\cite{St4},  \cite{ClarFartSimsTomf}), associated to  Hausdorff ample groupoids, can be seen as  partial skew inverse semigroup rings
(see also \cite{Demeneghi}). More information on partial actions, partial representations and their applications may be found   in  Exel's book \cite{E6} and in   \cite{D3}. In particular,  the former reference contains detailed partial crossed product descriptions of  graph $C^*$-algebras and  the Wiener-Hopf $C^*$-algebras associated to  quasi-lattice ordered groups.

Among the theoretic developments two partial group cohomology theories were introduced: one in \cite{DKh} based on partial actions and another one in \cite{A-A-R} based on partial representations. Despite the fact that partial actions and partial representations are related notions, the two cohomology theories have little in common. In particular, 
the cohomology in  \cite{A-A-R}  deals with an abelian category, whereas the category of partial modules from   \cite{DKh} is not even additive. The cohomology from  \cite{DKh} turned out to be useful to extensions of semilattices of groups by groups, to a Chase-Harrison-Rosenberg exact sequence for partial Galois extensions,   to partial projective group representations and to  reduced $C^*$-crossed products   (see \cite{D3}). In particular, it was shown in \cite{DoSa} that each group component of the partial Schur multiplier is a partial cohomology group. Furthermore, partial $2$-cocycles  appeared in \cite{KennedySchafhauser} as certain obstructions  in the study of the  ideal structure of the reduced crossed product of a $C^*$-algebra by a global action.  In addition, it was extended from groups to groupoids in \cite{NysOinPin3} and used to 
classify the equivalence classes of certain groupoid graded rings.  
Partial  actions of Hopf algebras were introduced in \cite{CJ} and were reformulated in 
\cite{HV} and \cite{SV} with the aim of obtaining a richer theory of partial actions of algebraic groups. The partial cohomology theory developed in \cite{DKh} stimulated its Hopf theoretic treatment  in \cite{BaMoTe},  where a natural generalization of Sweedler's cohomology  to the partial action setting was given.

The partial modules in \cite{DKh} are unital partial actions of groups on commutative monoids, whereas the authors in  \cite{A-A-R} deal with modules over the partial group algebra  $\K_{par} G.$ The latter is an algebra
which governs the partial  $\K$-representations of a group $G,$ and it has  as a $\K$-basis the  above mentioned  semigroup 
$S(G),$ introduced by R. Exel in \cite{E1} to deal with partial actions and partial representations. R. Exel defined $S(G)$ by generators and relations and gave a canonical form of the elements of $S(G)$ which he used  to show that $S(G)$ is an inverse monoid \cite{E1}. Later in  \cite{KL} J. Kellendonk and M. Lawson proved that  $S(G)$ is isomorphic to the Szendrei expansion of $G,$ and, as a consequence of a result by M. Szendrei \cite{Sze},  $S(G)$ is isomorphic to the Birget-Rhodes expansion of $G.$ The information on $S(G)$ given in \cite{E1} and 
\cite{KL} are important technical tools when dealing with partial linear and partial projective group 
representations, as well as with the partial cohomologies.

The algebraic study of   $\K_{par} G$ was initiated in  \cite{DEP}. In particular, using a groupoid $\Gamma (G)$
associated to $G,$ a result on the structure of 
   $\K_{par} G$ was obtained assuming that  $G$ is a finite group. The latter fact  permits one to make conclusions about the  partial representations of finite groups. 
   The connection of partial representations of $G$ with the groupoid $\Gamma(G)$ also inspired the development of a theory of partial representations of Hopf algebras in \cite{ABV1}, 
   \cite{ABV2}, where it was shown that partial representations corresponds to left modules over a Hopf algebroid. 
   The groupoid technique was further used in \cite{DZh} to give a structure result on the indecomposable/irreducible  finite dimensional partial representations of arbitrary groups.
Furthermore, it was shown in \cite{DE} that $\K_{par} G$ is isomorphic to the  partial skew group ring $B \rtimes _{\tau} G ,$ where $\tau $ is a partial action  of $G$ on the subalgebra $B$ of  $\K_{par} G$  generated by the idempotents of $S(G).$

 It is very natural  to consider a (co)homology theory for $G$ using   $\K_{par} G$-modules.  A key point in 
  \cite{A-A-R} is to choose as the ``trivial'' module the    $\K_{par} G$-module structure on $B$ given by the partial representation $ G\to {\rm End}_{\K} (B), $ which corresponds to the partial action $\tau $ 
(see Subsection~\ref{subsec:kparG}). Then the $n$th partial cohomology group $H_{par}^n(G,M)$ of $G$ with values in a left    $\K_{par} G$-module 
$M$ is defined  in \cite{A-A-R} by means of the corresponding ${\rm Ext}$ functor (see Subsection~\ref{subsec:introhomo}).  The authors  of   \cite{A-A-R}  prove the existence of a spectral sequence relating the 
 Hochschild cohomology of  the partial skew group ring  $\A\rtimes _\tau G$ with   $H_{par}^n(G, - )$ and the
 Hochschild cohomology of $\A .$ The globalization problem for $H_{par}^n(G,M)$ was investigated in \cite{DKhS2} in the case when  $M$ is an algebra, whose     $\K_{par} G$-module structure comes from a unital partial action of $G$ on $M.$ For more information    around   $\K_{par} G $ the reader is referred to \cite{D3} and    \cite{E6}.    In particular, the difference  between the cohomologies in \cite{A-A-R} and  \cite{DKh} is discussed in \cite{D3} and \cite{DKhS2}.

In this paper  we introduce homology groups for partial representations as the left derived functors of the tensor functors 
    $ - \otimes_{\kpg} B$, $B \otimes_{\kpg} - $, and we
  study homological and cohomological properties of the  partial group algebra $\kpg$ using ring theoretic homological methods.

Let $G$ be an arbitrary group, $\K$ be a field  and $V$ be an irreducible (resp. indecomposable) left $\K_{par} G$-module, which is  finite dimensional over $\K .$ By \cite[Thm.~2.3]{DZh}  there is a unique connected component $\Delta $ of $\Gamma (G)$ with a finite number of vertices, an elementary   bimodule $_{\K_{par} G} W _{\K H}$ related to $\Delta$  and a left   irreducible (resp. indecomposable) $\K H$-module $U,$ finite dimensional over $\K , $ such that
$V  \simeq W \otimes_{\K H} U$ as left  $\K_{par} G$-modules. Here $H$ is the stabilizer in $G$ of a fixed vertex of the connected component $\Delta$.

 Our main goal is to establish the following result.

\medskip
 
 {\bf Theorem A} {\it Let $G$ be an arbitrary group, $\Delta $ be a connected component of  the groupoid $\Gamma (G)$ with a finite number of vertices and $W$ be the elementary $\kpg$-$\K H$-bimodule related to $\Delta$. Let $U$ be an arbitrary left $\K H$-module.  Then 
 $$
 H^i_{par}(G, W \otimes_{\K H}U ) \simeq H^i(H, U) 
 ~ \hbox{ 
 and } ~
 H_i^{par}(G, W \otimes_{\K H}U ) \simeq H_i(H, U). 
 $$
In particular, if $V$ is an irreducible (resp. indecomposable) left $\K_{par} G$-module, which is  finite dimensional over $\K $ and as above $V  \simeq W \otimes_{\K H} U$ as left  $\K_{par} G$-modules, then $$H_i^{par}(G,V) \simeq H_i(H, U) ~ \hbox{ and  } ~ H^i_{par}(G,V) \simeq H^i(H, U).$$}

\medskip

By definition for a left $\kpg$-module $V$ we have $ H_{par}^n(G, V) = Ext^n_{\kpg} (B, V)$ where $B$ is considered as a  left   $\kpg$-module and $ H^{par}_n(G, V) = Tor_n^{\kpg} (B, V)$, where $B$ is considered as a     right $\kpg$-module (see Subsection~\ref{subsec:introhomo}).

Theorem A follows from Theorem~\ref{general1} and Theorem~\ref{general2}. Its  proof   needs  homological and cohomological versions of the Shapiro Lemma that we  give in Section \ref{Shapiro-section}.  Their  use  in the proof  of  Theorem A  requires  
$B \otimes_{\K_{par} G} W$
 to be   isomorphic to the trivial right $\K H$-module $\K$ and $\K \Delta$  to be   a flat left $\kpg$-module  (see  Proposition \ref{trivial} and Theorem \ref{final-flat}).    Proposition~\ref{trivial} is a  consequence of Theorem~\ref{isomorphism.Bdelta} whose proof is separated in
 Section~\ref{section.lemma7.7}. 
 In the case when $G$ is finite by \cite{DEP} $\kpg \simeq \K \Gamma \simeq \oplus \K \Delta$, hence $\K \Delta$ is a projective left $\kpg$-module. Furthermore, we show in Section \ref{projective-section} that  if $A$ is a vertex of a component $\Delta$ with finitely many vertices and $G \setminus A$ is infinite, then  $\K \Delta$ is not a projective left $\kpg$-module.
 Moreover,
  if $G$ is an  infinite group, then for the connected component $\Delta$ with unique vertex $G$, $\K \Delta = \K G$ is not projective as a $\kpg$-module.

 In Section \ref{section-bimodule} we  establish  the following  consequences  of Theorem A for finite groups $G$.
 
 \medskip{\bf Corollary B} {\it 	Suppose  that  $G$ is a finite group. Then for $i \geq 0$
 	$$
 	H_i^{par} (G, B) \simeq \oplus_{\Delta} H_i(H, \K) \ \hbox{ and } \ H^i_{par} (G, B) \simeq \oplus_{\Delta} H^i(H, \K),
 	$$
 	where the sum is  taken over the connected components $\Delta$ of the groupoid $\Gamma(G)$ and $H$ is the stabilizer of a chosen vertex of $\Delta$ (thus   it  is defined up to conjugation).}

 \medskip
 
 We define the {\it  partial cohomological dimension} $cd^{par}_{\K}(G)$ of $G$ (over $\K$) as the projective dimension of $B$ as  a   left $\kpg$-module i.e.  $$ cd^{par}_{\K}(G) = \max \{ i ~| ~ H^i_{par}(G, V) \not= 0 \hbox{ for some } \kpg \hbox{-module } V\}.$$
 Note that {\it the cohomological dimension} $cd_{\K}(G)$ of $G$ (over $\K$) is the projective dimension of the trivial left $\K G$-module $\K$ i.e. 
 $$ cd_{\K}(G) = \max \{ i ~| ~ H^i(G, M) \not= 0 \hbox{ for some } \K G \hbox{-module } M\}.$$
 
 As a corollary of Theorem A we prove the following result.
 
 \medskip
 {\bf Corollary C} {\it For any group $G$ we have $cd^{par}_{\K}(G) \geq cd_{\K}(G).$ In particular, if  $char(\K) = p > 0$  and $G$ has $p$-torsion then    $cd^{par}_{\K}(G) = cd_{\K} (G) = \infty$.}

\medskip
We believe that the following stronger version of Corollary C holds.

\medskip
{\bf Conjecture D} {\it For any group $G$ we have $cd^{par}_{\K}(G) = cd_{\K}(G).$}

\medskip In the particular case when $G$ is a free group   we have that   $cd_{\K}(G)=1$  and Conjecture D is equivalent to  the following:

\medskip

{\bf Conjecture E} {\it  Let $G$ be a free group  and $R = \K_{par} G.$  Then  the kernel $IG$ of the augmentation map $\epsilon : R \to B$ is a projective left $R$-module.}

\medskip
We prove in Section \ref{section-Z} that Conjecture E holds for $G = \Z,$  so that  $cd^{par}_{\K}(G) = cd_{\K}(G) =1.$ It is interesting to note that in this case $IG$ is not a free $\kpg$-module but only a projective one.

\section{Notation and preliminary results}\label{sec:notations}

\subsection{The  partial  group algebra $\K_{par} G$}\label{subsec:kparG} 

 Let $G$ be a group and $\K$ be a field. We recall some well known definitions  and facts. 

\begin{definition}
A {\bf partial representation } of $G$ into a  unital (associative)  $\K$-algebra $\A$  is a map ${\pi : G \to \A}$ such that for 
every $g,h \in G$ we have

1) $\pi(g) \pi(h) \pi(h^{-1}) = \pi(gh) \pi(h^{-1})$;

2) $\pi(g^{-1}) \pi(g) \pi(h) = \pi(g^{-1}) \pi(gh)$;

3) $\pi(1) = 1_{\A}$, where $1$ is the neutral element of $G.$
\end{definition}

If  $\A=End_{\K}(V),$ where $V$ is a vector space over $\K ,$ then we say that $\pi $  is a   partial representation  of 
$G$ on $V.$

  A  concept closely related to partial representations is that of a partial action:

\begin{definition} A partial action $\alpha$ of $G$ on a $\K$-algebra $\A$ is given by a collection $\{ D_g \}_{g \in G}$ of 
ideals in $\A$ and a family $\{ \alpha_g : D_{g^{-1}} \to D_g  \}_{g \in G}$ of non-necessarily unital algebra isomorphisms satisfying:
	
	1)  $D_e = \A$ and $\alpha_e = id_{\A}$;
	
	2) $ \alpha_h( D_{h^{-1}} \cap D_{(gh)^{-1}}) = D_h \cap D_{g^{-1}}$;
	
	3) if $ y \in  D_{h^{-1}} \cap D_{(gh)^{-1}}$ then $\alpha_g \alpha_h(y) = \alpha_{gh}(y)$.
\end{definition} 

A partial action is said to be {\bf unital} if each ideal $D_g $ is a unital algebra, i.e. $D_g = 1_g \A$ for each $g\in G,$  where $1_g$ is a central idempotent of $\A.$ To a unital partial action $\alpha $ of $G$ on $\A$     one may associate two partial representations. One of them is the map  $ G \to {\rm End} _{\K} (\A)$ given by 
\begin{equation}\label{FirstParRep} 
g \mapsto [ a \mapsto  \alpha _g(a1_{g^{-1}})] , \;\;\;\;\; g\in G, a\in \A. 
\end{equation}   The other one involves  the partial  skew group algebra. The latter is defined 
 as follows.  

\begin{definition} Suppose there is a  unital partial  action  $\alpha $  of $G$ on a $\K$-algebra $\A.$  Then there is an 
associative   partial skew group ring (also called the partial smash product),  denoted by  $\A \rtimes_{\alpha} G$, where
	$$
	\A \rtimes_{\alpha} G = \oplus_{g \in G} D_g \# g
	$$
	 and $$(a { 1_g} \# g) (b { 1_h} \# h) = a \alpha_g(b { 1_h 1_{g^{-1}}  })  \# gh \hbox{ for } a,b \in \A.$$
\end{definition}
  The second partial representation associated to $\alpha $ is  the map $G \to  \A \rtimes_{\alpha} G$ given by 
$g \to {  1_g} \# g$.

There is a bijective correspondence between the partial $\K$-representations of $G$ and the 
 $\K$-representations 
of the partial group algebra $\K_{par} G,$ whose definition is as follows.
 
\begin{definition}
The {\bf partial group algebra } $\K_{par} G$ is the $\K$-algebra with a generating set $\{ [g] \mid g \in G\}$ subject to the relations

1) $[g][h][h^{-1}] = [gh] [h^{-1}]$;

2) $[g^{-1}] [g] [h] = [g^{-1}] [gh]$;

3) $[1] = 1_{\kpg}$,

 for all $g,h\in G.$
\end{definition}  \noindent For simplicity, we  will also denote the identity of $\kpg$  by $1$ instead of $1_{\kpg}$.    Evidently, 
\begin{equation}\label{basicParRepr}
G\ni g \mapsto [g] \in   \K_{par} G
\end{equation}  is a partial representation.

Each partial representation  $\pi : G \to \A$ gives rise to a unital partial action on a commutative subalgebra of 
$\A$ \cite[Lemma 6.5]{DE}. We specify this for  the case of the partial representation \eqref{basicParRepr}. It is an easily
 seen and known fact that the elements $e_g = [g][g^{-1}]$ are pairwise commuting idempotents of   $\K_{par} G$ such that 
$$[g]e_h = e_{gh} [g] \hbox{ and } e_1 = 1,$$ 
for all $g,h\in G.$ 
Let
\begin{equation}\label{subalgB}
 B =_{\text{alg}}\langle e_g\, \mid \,  g\in G\rangle,
\end{equation} be the commutative subalgebra of     $\K_{par} G,$ generated by the elements $\{ e_g \}_{ g\in G}$
and let $D_g = e_g B.$ Then the partial representation  \eqref{basicParRepr} induces a  unital  partial action $\tau $
of $G$ on $B$  given by the family of isomorphisms $\tau_g : D_{g^{-1}} \to D_g$ defined by
	 \begin{equation}\label{basicParAc}
\tau_g(e_{g^{-1}} e_{h_1} \ldots e_{h_n}) = [g] e_{g^{-1}} e_{h_1} \ldots e_{h_n} [g^{-1}] = e_g e_{g h_1} \ldots e_{g h_n}.
\end{equation}

\begin{lemma} \cite[Theorem 6.9]{DE}\label{A-A-Rlema} There is an isomorphism $$ \varphi : \kpg  \to  B \rtimes_{\tau} G,$$
	given by
	$$
	\varphi ([g_1] \ldots [g_n]) = e_{g_1} e_{g_1g_2} \ldots e_{g_1 \ldots g_n} \# g_1 \ldots g_n
	$$ and
	$$
	\varphi^{-1} (e_{g} e_{h_1} \ldots e_{h_n} \# g)= e_g e_{h_1} \ldots e_{h_n}  [g].$$
	\end{lemma}

 It is an immediate consequence of Lemma~\ref{A-A-Rlema} that $\K_{par} G$ is a $G$-graded algebra:

$$\K_{par} G = \oplus_{g \in G} B_g \hbox{ and } B_g B_h \subseteq B_{gh},
$$ where $B_g = D_g \# g$. It is readily seen that  each $D_g$  is the $\K$-vector subspace of $B$ generated by $[h_1] \ldots [h_k]$ where $k \geq 1$ and $h_1 \ldots h_k = g.$

According to \eqref{FirstParRep}  the unital partial action \eqref{basicParAc}   gives rise to  the partial representation  
 $\pi : G \to End_{\K}(B)$ given by  $\pi(g)(x) = [g] x [g^{-1}]$ for $g \in G, x \in B$. This endows $B$ with a structure a left $\kparg$-module: $$^{[g]} b = [g] b [g^{-1}].$$ Symmetrically, we shall consider $B$ as a  right $\kparg$-module defined by 
$$  b^{[g]}= [g^{-1}] b [g].$$

\subsection{Introduction to homological algebra}\label{subsec:introhomo}

 We refer the reader to Rotman's book \cite{Rotman} for the definitions  of the derived functors $Tor$ and $Ext$.
For a fixed associative ring $R$ we consider the derived functors $Tor^R_i( - , - )$ of the functor  tensor product $ - \otimes_R -$ and the derived functors $Ext^i_R( - , - )$ of the functor $Hom_R( - , - )$. Then for a left $\kpg$-module $V$ and a right $\kpg$-module $M$ the $n$-th  partial homology 
groups
 of $G$ with coefficients in $M$ and $V$ are 
$$ H^{par}_n(G, M) = Tor_n^{\kpg} (M, B)  \hbox{ and } H^{par}_n(G, V) = Tor_n^{\kpg} (B, V).$$

 Note that $B$ in $ Tor_n^{\kpg} (M, B)$ is a left $\kpg$-module and that $B$ in $Tor_n^{\kpg} (B, V)$ is a right $\kpg$-module.
Similarly the $n$-th  partial cohomology groups
of $G$
with coefficients in $M$ and $V$    are 
 $$H_{par}^n(G, M) = Ext^n_{\kpg} (B, M) \hbox{ and } H_{par}^n(G, V) = Ext^n_{\kpg} (B, V),$$
 where $B$ in $Ext^n_{\kpg} (B, M)$ is a right $\kpg$-module and $B$ in 
 $Ext^n_{\kpg} (B, V)$ is a left $\kpg$-module.
 
  A fundamental result that we will need later on is the existence of  the  long exact sequence in homology associated to a short exact sequence. Suppose $$0 \to V_1 \to V \to V_2 \to 0$$ is a short exact sequence of left $S$-modules for some associative ring $S$. Then there is a long exact sequence \cite[Cor.~6.30]{Rotman} 
 $$
 \ldots \to Tor_n^{S}(M, V_1) \to Tor_n^{S}(M, V) \to  Tor_n^{S}(M, V_2) \to Tor_{n-1}^{S}(M, V_1)  $$ $$ \to  \ldots \to Tor_1^{S} (M, V_2) \to  M \otimes_{S} V_1 \to  M \otimes_{S} V \to   M \otimes_{S} V_2 \to 0.$$
Similarly, for a short exact sequence of right $S$-modules
$$0 \to M_1 \to M \to M_2 \to 0$$ there  is a long exact sequence
$$
\ldots \to Tor_n^{S}(M_1, V) \to Tor_n^{S}(M, V) \to  Tor_n^{S}(M_2, V) \to  Tor_{n-1}^{S}(M_1, V)$$ $$ \to  \ldots \to Tor_1^{S} (M_2, V) \to  M_1 \otimes_{S} V \to  M \otimes_{S} V \to   M_2 \otimes_{S} V \to 0.$$
 We observe that if $$ \ldots \to W_1 \to W_2 \to W_3 \to W_4 \to \ldots$$ is a  long exact sequence of $S$-modules with $W_1 = 0 = W_4$, then 
 $$ \hbox{  the middle map } W_2 \to W_3 \hbox{ is an isomorphism}.$$We will apply this argument several times later on together with the fact that for a right $S$-module $A$ and a left $S$-module $B$ the functors $Tor_i^S(A, - )
$ and $Tor^S_i( - , B)$ vanish on projective $S$-modules; in particular 
$$Tor_i^S(A,S) = 0 = Tor^S_i(S , B).$$

\subsection{On idempotents and projective modules}

In this subsection $R$ is  a unital associative ring.
The following lemma is well known but for completeness we give a proof. 
\begin{lemma} \label{idempotents} 
	Let $e_1, \ldots, e_n \in R$ be idempotents that commute with each other. Then  
	$$
	Re_1 + R e_2 + \ldots + R e_n = R e,$$ where 
	$$e = e_1 + (1 - e_1) e_2 + \ldots +  ( 1- e_1) \ldots (1 - e_{n-1}) e_n
	$$ and	
	$R / (R e_1 + \ldots  + Re_n)$ is a projective $R$-module.
\end{lemma}

\begin{proof} 
	Note that $e_1$, $(1 - e_1) e_2$, $(1 - e_1) (1 - e_2) e_3$, $\ldots, ( 1- e_1) \ldots (1 - e_{n-1}) e_n$ are pairwise orthogonal  idempotents (i.e. with mutual product 0). Then
	$e$ is an idempotent and 
	$Re_1 + R e_2 + \ldots + R e_n = R e.$ 	Since $R = R e \oplus R ( 1- e)$ we conclude that  $$R / (R e_1 + \ldots  + Re_n) = R / R e \simeq R ( 1 - e)$$ 
is a projective $R$-module.
\end{proof}

 \section{Partial group homology  versus ordinary homology} \label{Shapiro-section}

 \begin{theorem}\label{HomolShapiro} ({\bf Homological Shapiro lemma for partial actions}) Let $W$ be a 
 $\kpg$-$\K H$-bimodule for some groups $G$ and $H$   such that $W$ is flat as a right $\K H$-module,  $W$ is  flat  as a left $\K_{par} G$-module and  $B \otimes_{\K_{par} G} W$ is isomorphic to the trivial right $\K H$-module $\K$. Then, for any  
 left $\K H$-module $U$,  considering $W \otimes_{\K H} U$ as a   left $\K_{par} G$-module,  
 we have
 	$$
 	H_i^{par}(G, W \otimes_{\K H} U) \simeq H_i(H, U), \;\;\;  i \geq 0.
 	$$
 \end{theorem}
 \begin{proof} Let
 	$$
 	{\mathcal P} : \ldots \to P_i \to P_{i-1} \to  \ldots \to P_0 \to U \to 0
 	$$
 	be a free resolution of $U$ as a left $\K H$-module. Since $W$ is a  flat right $\K H$-module
 	the functor $W \otimes_{\K H} - $ is exact and hence we obtain an exact  complex
 	$$
 	{\mathcal S} = W \otimes_{\K H} {\mathcal P} : \ldots \to S_i \to S_{i-1} \to \ldots \to S_0 \to W \otimes_{\K H} U \to 0 , 
 	$$
 	where $S_i = W \otimes_{\K H} P_i $. We view each $S_i$ as a left $\K_{par} G$-module via the left action of $\K_{par} G$ on $W$. Since each $P_i$ is a free $\K H$-module we have that \begin{equation} \label{S} S_i = W \otimes_{\K H} P_i = W \otimes_{\K H} (\oplus \K H) = \oplus W , \end{equation}
 	hence
 	$$
 	S_i \hbox{ is a  flat  left }\K_{par}G-\hbox{module}, 
 	$$
 	 	so ${\mathcal S}$ is a  flat  resolution of the left $\K_{par} G$-module $W \otimes_{\K H} U$.  By \cite[Thm.~7.5]{Rotman} not only projective but flat resolutions can be used to calculate $Tor_i^{\kpg}(B, -)$.  Then
 	$$
 	H_i^{par}(G, W \otimes_{\K H} U) = Tor_i^{\K_{par}G}(B, W \otimes_{\K H} U) 
 	= H_i(B \otimes_{\K_{par} G} {\mathcal S}^{del}), 
 	$$
 	where the upper index $del$ means the deleted resolution,  i.e. the complex in dimension $-1$ is substituted with 0.
 	Note that 
 	 $$B \otimes_{\K_{par} G} S_i = B \otimes_{\K_{par} G} (W \otimes_{\K H} P_i)  \simeq   (B \otimes_{\K_{par} G} W) \otimes_{\K H} P_i ,$$ hence
 	$$
 	H_i(B \otimes_{\K_{par} G} {\mathcal S}^{del}) \simeq H_i((B \otimes_{\K_{par} G} W) \otimes_{\K H} {\mathcal P}^{del} ).
 	$$
 	Since $B \otimes_{\K_{par} G} W$ is the trivial right $\K H$-module $\K$, 
 	$$
 	H_i((B \otimes_{\K_{par} G} W) \otimes_{\K H} {\mathcal P} ) \simeq H_i(\K \otimes_{\K H} {\mathcal P}^{del}  ) = H_i(H, U).
 	$$
 \end{proof}
 
  For a $\K_{par}G$-$\K H$-bimodule $W$ and a right $\K H$-module $U$ consider $Hom_{\K H} (W, U)$ as a right $\kpg$-module in the following way : for $f \in Hom_{\K H} (W, U)$ and $r \in \kpg$ we have $ fr \in Hom_{\K H} (W, U)$ defined by $fr (w) = f ( rw)$.
 
 \begin{theorem}\label{CohomolShapiro}({\bf Cohomological Shapiro lemma for partial actions}) 
 	Let $W$ be a $\K_{par}G$-$\K H$-bimodule for some groups $G$ and $H$  such that $W$ is  projective  as a right $\K H$-module, 
	$W$ is  flat as a left $\K_{par} G$-module and  $B \otimes_{\K_{par} G} W$ is isomorphic to the trivial right $\K H$-module $\K$. 
 	Then for any   right $\K H$-module $U$, 
 	considering  $Hom_{\K H} (W, U)$ as  a right  $\K_{par} G$-module, we have
 	$$
 	H^i_{par}(G, Hom_{\K H} (W, U)) \simeq H^i(H, U),\;\;\;  i \geq 0.
 	$$
 \end{theorem}
 
 \begin{proof} Let
 	$$
 	{\mathcal E} : 0 \to U \to E^0 \to E^1 \to E^2 \to \ldots
 	$$
 	be an injective resolution of $U$ as a right  $\K H$-module.
 	Since $B \otimes_{\K_{par} G} W$ is isomorphic to the trivial right $\K H$-module $\K$ we have
 	\begin{align*}
 	H^i(H,U) & = Ext^i_{\K H} (\K, U)  = Ext^i_{\K H} ( B \otimes_{\K_{par} G} W , U) \\
 	& = H^i( Hom_{\K H}( B \otimes_{\K_{par} G} W, {\mathcal E}^{del} )).
 	\end{align*} 	 
 	By the Adjoint Isomorphism Theorem \cite[Thm.~2.75]{Rotman} we obtain
 	$$
 	Hom_{\K H}( B \otimes_{\K_{par} G} W, E^i) \simeq Hom_{\K_{par} G} (B, Hom_{\K H} (W, E^i)), 
 	$$
 	hence 
 	$$
 	H^i( Hom_{\K H}( B \otimes_{\K_{par} G} W, {\mathcal E}^{del} )) \simeq H^i( Hom_{\K_{par} G} (B, Hom_{\K H} (W, {\mathcal E}^{del})) ).
 	$$
 	We claim that  $ Hom_{\K H} (W, E^i)$ is injective as a right $\K_{par} G$-module. This is equivalent 
 	to  the exactness of the functor $Hom_{\K_{par} G} ( - , Hom_{\K H} (W, E^i))$. Consider again  the adjoint isomorphism 
 	$$
 	Hom_{\K_{par} G} ( - , Hom_{\K H} (W, E^i)) \simeq Hom_{\K H} ( - \otimes_{\K_{par}G} W, E^i) .
 	$$
 	Note that the right hand side is the composition of two exact functors 
 	$ - \otimes_{\K_{par} G} W$ and $Hom_{\K H} ( - , E^i)$, because $W$ is  flat as a left $\K_{par} G$-module and $E^i$ is injective as a right  $\K H$-module.
 	As a composition of two exact functors is exact, we obtain that $Hom_{\K H} (W, E^i)$ is injective as a right  $\K_{par} G$-module. 
 	Consequently, since $W$ is a projective right $\K H$-module,  $Hom_{\K H} (W, - )$ is an exact functor and	$Hom_{\K H} (W, {\mathcal E})$ is an injective resolution of the right $\K_{par} G$-module $Hom_{\K H} (W, U)$.
Then
 	\begin{align*}
 	H^i (H,U) & \simeq H^i( Hom_{\K_{par} G} (B, Hom_{\K H} (W, {\mathcal E}^{del} )) \\
 	 & \simeq Ext_{\K_{par}G}^i(B, Hom_{\K H} (W, U)))\\
 	 & = H^i_{par} (G, Hom_{\K H} (W, U)).
 	\end{align*}
 \end{proof}

\section{Applications: the bimodule $W$} \label{section-bimodule}
 	We are going to show the existence of a bimodule $W$ as in Theorem~\ref{HomolShapiro} and Theorem~\ref{CohomolShapiro}. First we recall some notions from  \cite{DEP} and \cite{DZh}. Given an arbitrary groupoid $\Gamma $ and a field $\K,$ the groupoid algebra $\K \Gamma $ is defined as the vector space with basis formed by  the morphisms of $\Gamma ,$ which are multiplied by the rule

\[\gamma_1\cdot\gamma_2=\left\{
\begin{array}{ll}\gamma_1 \circ \gamma_2,&\text{if}\ \gamma_1 \circ \gamma_2 \;\text{exists in }  \Gamma, 
\\ 0,  &  \text{otherwise}.
\end{array}
\right.\] As it is usual for functions,  we apply morphisms from right to left, so that in the composition 
$\gamma_1 \circ \gamma_2$  first  $\gamma _2$  acts and then $\gamma _1.$

For a group  $G$ consider the groupoid $\Gamma (G)$ whose objects (seen as vertices) are subsets of $G$ which contain $1$ and whose morphisms are  pairs $(A,g)$ where $A$ is a subset of $G$  which contains  $1$ and  $g\m.$ Then both $A$ and $gA$ are    objects in $\Gamma (G)$ and the morphism   $(A,g)$ can be seen as an arrow from $A$ to $gA$, which is interpreted as multiplication by $g$ from the left. The product $(B,g') \circ (A,g)$ is defined in $\Gamma (G)$ if and only if $B =gA, $ i.e. when the range $gA$ of $(A,g)$ coincides with the domain $B$ of $(B,g').$ In this case  $(B,g') \circ (A,g) = (gA,g') \circ (A,g) = (A, g' g).$  We have the decomposition of the groupoid algebra 
$$
 \K \Gamma(G) = \oplus \kd,$$ into a direct sum of two-sided ideals, where $\Delta$ runs over the connected components of $\Gamma (G).$  

Let now $\Delta $ be a connected component of $\Gamma (G)$ whose set of vertices 
${\mathcal V}_{\Delta }$ is  finite. 
By \cite[Theorem 2.2]{DZh} the map 
\begin{equation}\label{lambdaH}
  \lambda_{\Delta}(g) = \sum_{A \in {\mathcal V}_{\Delta }, g^{-1} \in A} (A, g)
\end{equation}
   is a partial representation  $G \to \K \Delta ,$ where the sum is taken  over all vertices $A$ of $\Delta $ which contain $g\m.$ By the universal property of $\K _{par} G,$ the map $\lambda_{\Delta}$ 
  extends to a homomorphism of $\K$-algebras  
$$\K _{par} G \to \kd, \; [g]\mapsto \lambda_{\Delta}(g),$$ which, with a slight abuse of notation, will be denoted by the same symbol $\lambda_{\Delta}$.  
We consider $\kd$ as a left $\kpg$-module by means of $\lambda_\Delta$, i.e.,
 \begin{equation}\label{KparG-Delta}
[g] \cdot a = \lambda_{\Delta } (g) a
 \end{equation}
for $a \in \kd$ and $g \in G$.

Let $n$ be the number of vertices in   $\Delta $ and  fix a vertex $A\in {\mathcal V}_{\Delta }.$ Then there exist $g_1, g_2, \ldots , g_n \in G,$ with $g\m_i \in A,$ such that 
$ {\mathcal V}_{\Delta } = \{g_1 A, g_2A, \ldots, g_nA \}.$  It is convenient to  assume that $g_1=1.$
Let $H$ be the stabilizer of  $A,$ i.e. 
$$H= \{h \in G : hA =A \}.$$ Then $H$  is a subgroup of $G$ and the groupoid algebra $\kd$ is isomorphic to the  algebra $M_n(\K H)$ of all $n\times n$-matrices with entries in the group algebra $\K H.$ The isomorphism is given as follows. Let $(g_iA , g)$ be a morphism in $\Delta$ whose range is $g g_i A = g_j A$ for some  $j \in \{1, \ldots , n \}.$ Notice that $g\m_j g g_i =h_i\in H$ and write 
$g= g_j h_i g\m_i.$ Then the desired isomorphism is given by
\begin{equation}\label{isomorphism.kdelta.matrices}
\eta : \kd \to M_n (\K H ),  \ (g_iA , g) \mapsto  E_{ji}(h_i), 
\end{equation}
 where $E_{ji}(h_i)$ stands for the $n\times n$-matrix  whose unique non-zero entry $h_i$  is placed in the position $(j,i).$

  For $g \in G$ denote by $M_g$ the matrix $\lambda_{\Delta}([g])$ after identifying $\kd$ with $M_n(\K H)$.  Then  
\begin{equation}\label{elementaryRepr}
M_g=  \sum _{g _iA\ni g\m} E_{ji}(h_i),
\end{equation} in which the sum is taken  over all vertices $g_iA$ of $\Delta$ which contain $g\m.$ 
 Thus $M_g$ is a monomial  matrix over $H$, 
 in the sense that each row and each column of $M_g$ contains at most one non-zero entry, which is an element of $H.$ Denote by  $M_{g}^{\ast}$ the matrix obtained from $M_{g}$ by transposition and replacement of each entry $h\in H$ by its inverse.  
  
  \begin{lemma} \label{action} 
  Let $G$ be an arbitrary  group and let $M_g$ be as defined above.   For every $g \in G$ we have that  $M_{g\m} = M_{g}^{\ast}.$  
  \end{lemma} 

\begin{proof} 
     While in (\ref{elementaryRepr}) $g _iA$ runs over all vertices of $\Delta$ which contain $g\m, $  the target vertices $g_jA = gg_iA$ of the arrows  $(g_iA,g)$ are all those vertices of $\Delta$ which contain $g.$
Evidently,   $(g_jA,g\m)$ is the inverse of   $(g_iA,g)$ in the groupoid $\Delta,$ and, consequently,   it follows from   (\ref{elementaryRepr})  that 
\[
 M _{g\m} = \sum _{g_jA\ni g} E_{ij}(h\m_i) =M^{\ast}_g,
 \]
as desired.
  
\end{proof}

\begin{lemma} \label{Bdelta}
Let $\Delta,$ $\lambda _{\Delta}$
 and $B$ be as above, with $G$ being an arbitrary   group. Then $\lambda _{\Delta } (B)$ is the subalgebra of $\kd$ 
 whose $\K$-basis is formed by 
the elements $(C, 1),$ where $C$ runs over $ {\mathcal V}_{\Delta }.$
\end{lemma}  
\begin{proof}
We need to show that $\{ (g_i A,1) ; i=1,2, \ldots, n \}$ is a basis for $\lambda_\Delta(B)$. 
First of all, we have the formula
\begin{equation}\label{formula.lambda-delta(e_g)}
\lambda_\Delta (e_g)  = \sum _{ g \in C, C \in {\mathcal V}_{\Delta }} (C, 1) =  \sum_{g_i A \ni g} (g_i A, 1)  
\end{equation}
since, by the definition of $\lambda_\Delta$ and the expression of the product in $\K \Gamma(G)$, 
\begin{align*}
\lambda_\Delta (e_g) & = \lambda_\Delta ([g])\lambda_\Delta ([g^{-1}])
 = \sum_{g_iA \ni g^{-1}, g_jA \ni g} (g_iA, g) (g_jA,g^{-1} ) = \sum_{g_jA \ni g} (g_j A, 1).
\end{align*}
Consequently, for every finite $X \subseteq G$,
\begin{equation}\label{soma}
\lambda_\Delta \left( \prod_{t \in X}  e_t\right) = \prod_{t \in X} \lambda_\Delta (e_t) =   \sum_{C\in {\mathcal V}_{\Delta }, \  C \supseteq X} (C, 1) ,
\end{equation}
assuming that this sum is zero if no such vertex $C $ exists. 

Fixing  now an arbitrary  $ C_0 \in {\mathcal V}_{\Delta }$ we shall show that $(C_0, 1) \in \lambda _{\Delta } (B).$ 

We shall proceed  using  induction on $d(C_0),$ where  the function $d:  \mathcal{V}_{\Delta} \to \mathbb{N}$ was defined in  \cite[p. 2571]{DdLP} as follows. If $C \in \mathcal{V}_{\Delta}$ is maximal, that is, if
\[ C \subseteq  \tilde{C}  \text{ for some } \tilde{C} \in \mathcal{V}_{\Delta}  \text{ implies } C = \tilde{C}, \]
we  set $d(C) = 0.$  Otherwise, we define
\[d(C) = \max\{ m \in \mathbb{N} \mid \exists D_1, \ldots, D_m \in \mathcal{V}_\Delta \text{, with } D_m \supsetneq \ldots \supsetneq D_1 \supsetneq C\}.\] 
 Since ${\mathcal V}_{\Delta }$ is finite, there is a finite subset $X_0 \subseteq C_0$ such that if $X_0 \subseteq C$ for some $C\in {\mathcal V}_{\Delta }$ then $C_0 \subseteq C$.
Indeed, let $\mathcal I \subseteq \{1, \ldots , n\}$  be the set of all indexes $j$ such that $C_0 \not\subseteq g_{j}A.$ 
For each $j\in \mathcal I$  choose an element $t_j \in C_0 \setminus g_{j}A.$ Then $X_0 = \{t_j \mid j \in \mathcal I  \}$ is the desired finite set.

If $d(C_0) = 0,$ then   (\ref{soma}) for $X = X_0$ 
 readily yields $(C_0,1) \in \lambda _{\Delta } (B),$ so  suppose that $d(C_0) >  0$ and $(C,1) \in  \lambda _{\Delta } (B)$ for all $C \in   \mathcal{V}_{\Delta} $ with $d(C) < d(C_0).$  Then we conclude from \eqref{soma} applied for $X = X_0$ that 

$$ (C_0, 1) = \lambda_\Delta \left( \prod_{{  t \in X_0} }  e_t\right)  -  \sum_{C\in {\mathcal V}_{\Delta }, \   C \supsetneq C_0 } (C, 1) \in  \lambda _{\Delta } (B), $$ completing our proof.

\end{proof}   

\begin{cor} \label{dimension.Bdelta} With the notation of Lemma~\ref{Bdelta}, we have that 
 $\dim_{ \K} (\lambda_\Delta(B)) = n$.
\end{cor}
\begin{proof}
It is clear that the linearly independent elements $(g_iA,1)$ generate the image of $B$ by $\lambda_\Delta$, and hence $\dim_{ \K}  (\lambda_\Delta (B) ) = n$. \end{proof}

The free right $\K H$-module  
  $$
  W = (\K H)^n = M_{n,1}(\K H),
  $$ viewed as  column matrices, possesses   a left $M_n(\K H)$-action given by matrix multiplication, transforming $W$ into an $M_n(\K H)$-$\K H$-bimodule. Then the $\K$-algebra homomorphism  
$$
  \lambda_{\Delta} : \K_{par} G \to \kd,
  $$ followed by the  isomorphism
  $$  \eta : \kd \to M_n(\K H),$$
endows $W$ with a structure of a  $\K_{par} G$-$\K H$-bimodule, called  {\it elementary}. The latter gives rise to an elementary partial matrix representation considered in \cite{DZh}.

We recall the next fact, whose statement  is a slight modification of Theorem 2.3 from \cite{DZh}.
 
 \begin{theorem}  \label{decomposition} 
 	Let $G$ be an arbitrary group and $V$ be an irreducible (resp. indecomposable) left $\K_{par} G$-module, which is  finite dimensional over $\K .$ Then there is a unique connected component $\Delta $ of $\Gamma (G)$ with a finite number of vertices, an elementary   bimodule $_{\K_{par} G} W _{\K H}$ related to $\Delta$  and a left   irreducible (resp. indecomposable) $\K H$-module $U,$ finite dimensional over $\K , $ such that
 	$V  \simeq W \otimes_{\K H} U$ as left  $\K_{par} G$-modules.   
 	\end{theorem} 

 \noindent  It is shown in   \cite[Theorem 2.2]{DZh} that  for each irreducible (indecomposable)  $\K_{par} G$-module $V,$ which has a finite dimension over $\K ,$ there exists a unique connected component $\Delta $ of $\Gamma (G),$ with a finite number of vertices, and a  $\kd$-module structure on $V$ such that   $_{\K_{par} G} V$ is obtained from $_{\kd } W$ $\otimes_{\K H} U$ by means of $  \lambda_{\Delta},$ so that  Theorem~\ref{decomposition} is a  direct consequence of \cite[Theorem 2.3]{DZh}.  Notice that the map  $$ M : G \to M_n(\K H), \;\;\; g \mapsto M_g=  \sum _{g_iA\ni g\m} E_{ji}(h_i),$$
is the elementary matrix representation afforded by the  bimodule  $_{\K_{par} G} W _{\K H}$.

 Given $d \in \ld(B)$ and 
$\sum_{\gamma \in \Delta} \alpha_{\gamma} \gamma \in \K \Delta $, define the action
\begin{equation}\label{right.kDelta.structure}
d \ract  (\sum_{\gamma \in \Delta} \alpha_{\gamma} \gamma ) = \sum_{\gamma \in \Delta} \alpha_{\gamma} \gamma^{-1} d \gamma. 
\end{equation}

Using Lemma \ref{Bdelta} it is easy to check that 
$\ld(B)$ is a right $\K \Delta$-module with respect to \eqref{right.kDelta.structure}: it suffices to check that $(tD,t^{-1})(C,1)(D,t) \in\ld(B)$, for any $(C,1), (D,t) \in \Delta$, and this product is either zero or equal to $(D,1) \in \ld(B)$.

The next fact is crucial for our construction of $W$   as in Theorem~\ref{HomolShapiro} and  Theorem~\ref{CohomolShapiro}, and since its proof is long it will be given in Section \ref{section.lemma7.7}. 
 
\begin{theorem} \label{isomorphism.Bdelta}
	Let $\Delta $ be a connected component of  $\Gamma (G)$ with a finite number of vertices and $\ld $ be as in  \eqref{lambdaH}. Then there is an isomorphism of right $\kd$-modules 
	\[
	\lambda_\Delta (B) \simeq B \otimes_{\K_{par} G} \kd,
	\]  where the left   ${\K_{par} G}$-module structure on $\kd $ is given by \eqref{KparG-Delta}.
\end{theorem}

Let $^*$ be the involution of $\kpg$ determined by $[g]^* = [g^{-1}]$.  Thus $e_g^{*} = ([g][g^{-1}])^{*} = [g^{-1}]^* [g]^{*} = [g] [g^{-1}] = e_g$ for every $g \in G$.
Note that $\kd$ also has a natural involution determined by 
\[
\left( \sum_{\gamma \in \Delta} \alpha_{\gamma} \gamma \right) \mapsto  \sum_{\gamma \in \Delta} \alpha_{\gamma} \gamma^{-1}
\]  
which, by a slight abuse of notation, will also be denoted by $^*$. It is clear from the equality (\ref{lambdaH})  which defines $\ld$ that $\ld(a^*) = \ld(a)^*$
  for every $a \in \kpg .$

  \begin{prop} \label{trivial} 
  Let $G$ be an arbitrary group, $\Delta $ be a connected component of  $\Gamma (G)$ with a finite number of vertices and $W$ be 
  the elementary $\kpg$-$\K H $-bimodule associated to $\Delta$ by means of the homomorphism $\ld$. Then   $B \otimes_{\K_{par} G} W \simeq \K$ as  right $\K H$-modules, where $\K$ is the trivial $\K H$-module and $B$ is a right $\K_{par} G$-module via $b^{[g]} = [g^{-1}] b [g]$  for $b \in B, g \in G$.
  \end{prop}

\begin{proof}
The algebra isomorphism $\kd \simeq M_{n} (\K H )$ endows $\kd$ with the structure of a right $\K H$-module.  
Consequently, as a right $\K H$-module, $\kd$ is isomorphic to the direct sum of $n$ copies of the free right $\K H$-module 
   $W,$ and we have   that 
\[
 B \otimes_{\kparg} \kd \simeq
B \otimes_{\kparg} W^n
\simeq 
(B \otimes_{\kparg} W)^n 
\]
 as right $\K H$-modules.  Then it follows from  Corollary \ref{dimension.Bdelta} and Theorem \ref{isomorphism.Bdelta}   that 
 \[ n = \dim_{ \K} (\lambda_\Delta(B)) = \dim _{\K}(B \otimes_{\kparg} \kd 
)  =  \dim_{ \K} ((B \otimes_{\kparg} W)^n),
 \]
  and therefore 
 $\dim_{ \K} (B \otimes_{\kparg} W) = 1$.  It remains to show that the right $\K H$ action on this module is trivial. 

Let $\{e_i ; i=1,2, \ldots, n\}$ be the canonical $\K H$-basis of
$W \simeq M_{n,1} (\K H )$.

\begin{claim}\label{1-tensor-ei}
$1 \otimes e_i h =  1 \otimes e_i$ in  $B \otimes_{\kparg} W$ for any  $h\in H$ and any  $ i=1,2, \ldots, n$. 
\end{claim}  In order to see this, fix $h\in H$ and recall that  $ {\mathcal V}_{\Delta } = \{g_1 A, g_2A, \ldots, g_nA \}$ where $g_1=1.$
 Notice that according to  \eqref{lambdaH} 
the element 
 $(g_iA,g_i hg_i^{-1})\in \Delta$ is a summand of 
$\ld(g_ihg\m_i)$ for each $i = 1, 2, \ldots, n$,  as $A\ni g\m_i$ implies $g_iA\ni g_ih^{-1}g\m_i.$ By \eqref{isomorphism.kdelta.matrices}, 
 $\eta (g_iA,g_i h g_i^{-1})= E_{ii}(h)$
 and $h$ is the $(i,i)$-entry of the matrix $\eta (\ld(g_ihg\m_i)).$   Recalling that $\eta (\ld(g_ihg\m_i))$ is monomial over $H$ 
and that the left ${\kparg}$-module structure on $W$ is given via the composition $\eta \circ \ld ,$ we obtain that 
$$
1 \otimes e_i h =  1 \otimes \eta (\ld(g_ihg\m_i)) e_i =   1^{[g_ihg\m_i]} \otimes e_i  $$ 
$$
  = [g_ih^{-1}g\m_i][g_ihg\m_i] \otimes e_i = 	e_{g_ih^{-1}g\m_i} \otimes e_i,
$$
recalling that if $b \in B$ and $r \in G$ then  $b^{[r]}=  [r^{-1}] b [r] \in B $ (see Proposition \ref{trivial}). 
  On the other hand 
 $$   
1 \otimes \eta (\ld(e_{g_ih^{-1}g\m_i})) e_i= 1 \otimes \eta (\ld([g_ih^{-1}g\m_i] [g_ihg\m_i])) e_i = 1^{[g_ih^{-1}g\m_i] [g_ihg\m_i]} \otimes e_i $$ $$= [g_ih^{-1}g\m_i] [g_ihg\m_i] [g_ih^{-1}g\m_i] [g_ihg\m_i] \otimes e_i = e_{g_ih^{-1}g\m_i} \otimes e_i.$$  Since by equality \eqref{formula.lambda-delta(e_g)} the element  $(g_iA,1)\in \Delta$ is a summand of
 $\ld(e_{g_ih^{-1} g\m_i}),$ the $(i,i)$-entry of the diagonal matrix    $\eta(\ld(e_{g_ih^{-1} g\m_i}))$ is $1,$ and thus 
$$1 \otimes e_i h =   1 \otimes  \eta (\ld(e_{g_ih^{-1} g\m_i}))   e_i  = 1\otimes e_i,$$ as claimed.

\begin{claim}
 $\{1 \otimes e_i \}$ is a $\K$-basis of $B \otimes_{\kparg} W$ for any  fixed $ i=1,2, \ldots, n$. 
\end{claim}

Indeed,  note that the elements $1 \otimes e_1, \ldots , 1 \otimes e_n$ generate $B \otimes_{\kparg} W$ as  a right $\K H$-module, since $b \otimes w =1 \otimes  \eta(\ld(b)) w$ for any $b \otimes w  \in B \otimes_{\kpg} W$.  It follows from Claim~\ref{1-tensor-ei}
that  $1 \otimes e_1, \ldots , 1 \otimes e_n$ generate $B \otimes_{\kparg} W$ as a $ \K$-vector space.

 Proposition 2.2 of \cite{DdLP} implies that $\ld : \kpg \to \kd$ is an epimorphism.  Hence, given $1 \leq i,j \leq n$,  we may write $E_{j,i}(1) =  \eta (\lambda_{\Delta }(a)) $ for some $a \in \kpg$, and it follows that
\begin{align*}
1 \otimes e_j & = 1 \otimes E_{j,i}(1) e_i = 1 \otimes \eta(\ld(a)) e_i =  1^a  \otimes e_i.
\end{align*} 
 For any $b \in B$, the element $\ld(b) \in \kd$ is a $\K$-linear combination of idempotents of the form $(C,1)$, $C \in \mathcal{V}_\Delta$, and therefore $\eta(\ld (b))$ is a diagonal matrix  with entries in $\K $ thanks to equality  \eqref{isomorphism.kdelta.matrices}. 
 In particular   $\eta(\ld (1^a))$ is  such a diagonal matrix. Therefore
\[ 1 \otimes e_j =   1^a \otimes e_i = 1 \otimes \eta(\ld (1^a)) e_i  =   \alpha_{j,i} (1 \otimes e_i)
\]
for some $\alpha_{j,i} \in \K$. 
Hence all elements $1 \otimes e_i$ are scalar multiples of each other and, given that those elements generate $B \otimes_{\kpg} W$ as a $\K$-vector space, at least one of the scalars $\alpha_{j,i}$ is nonzero; but this implies that all are nonzero. Due to the fact that  $\dim_{\K} B \otimes_{\kparg} W = 1$, any set $\{1 \otimes e_i \}$ is a $\K$-basis, proving the claim. 

 The fact that the right   $\K H$-action on $B \otimes_{\kpg} W$ is trivial
follows from Claim~\ref{1-tensor-ei}.
 \end{proof}

 The proof of the following result will be given in Section \ref{projective-section}.
 
 \begin{theorem} 
 	Let $G$ be an arbitrary group, $\Delta $ be a connected component of  $\Gamma (G)$ with a finite number of vertices and $A$ is finite.  Then $\K \Delta$ is projective as a left $\kpg$-module (via $\lambda_{\Delta}$). 
 	\end{theorem}
 
 Furthermore in Section \ref{projective-section} we will show that when $A$ is infinite $\K \Delta$ is not necessarily projective as a left $\kpg$-module (via $\lambda_{\Delta}$).
In Section \ref{flat-section} we will prove a more general result.

 \begin{theorem} \label{flat-left} Let $G$ be an arbitrary group, $\Delta $ be a connected component of  $\Gamma (G)$ with a finite number of vertices. Then $\K \Delta$ is flat as  a left $\kpg$-module (via $\lambda_{\Delta}$). 
\end{theorem}

\begin{theorem} \label{general1} Let $G$ be an arbitrary group, $\Delta $ be a connected component of  $\Gamma (G)$ with a finite number of vertices and $W$ be the elementary $\kpg$-$\K H$-bimodule related to $\Delta$. Let $U$ be an arbitrary left $\K H$-module. Then
		$$
	H_i^{par}(G,  W \otimes_{\K H} U) \simeq H_i(H, U), \;\;\;  i \geq 0.
	$$
\end{theorem}

\begin{proof} By the isomorphism  $ \kd \simeq M_n(\K H)$ given in \eqref{isomorphism.kdelta.matrices}, 
	$W$ is a direct summand of $ \kd$ as a $\kd$-module. Since $\kd$ is   flat  as a left $\K_{par} G$-module, we deduce that
	$$
	W \hbox{ is   flat  as a  left} \K_{par} G\hbox{-module}.
	$$
	Thanks to  Proposition \ref{trivial},   $B \otimes_{\K_{par} G} W$ is the trivial right $\K H$-module $\K$, and, because $W$ is also flat as a right $\K H$-module (being free), our statement follows by Theorem \ref{HomolShapiro}. 
\end{proof}

\begin{cor} Let   $\Delta$  be a connected component of the groupoid $\Gamma(G)$ such that $\Delta$ has finitely many vertices and   $V = W \otimes_{\K H} U$ be an irreducible (resp. indecomposable) left $\K_{par} G$-module with decomposition given by Theorem \ref{decomposition}. Then
	$$
	H_i^{par}(G, V) \simeq H_i(H, U), \;\;\;  i \geq 0.
	$$
	\end{cor}

 \begin{theorem} \label{general2} Let $G$ be an arbitrary group, $\Delta $ be a connected component of  $\Gamma (G)$ with a finite number of vertices and $W$ be the elementary $\kpg$-$\K H$-bimodule related to $\Delta$. Let $U$ be an arbitrary left $\K H$-module.  Then 
	$$
	H^i_{par}(G, W \otimes_{\K H}U ) \simeq H^i(H, U),\;\;\;  i \geq 0.
	$$
\end{theorem}

\begin{proof} 
We write $U_0$ for $U$ considered  as a right $\K H$-module by $u h = h^{-1} u$. Using
Theorem \ref{CohomolShapiro}, Proposition \ref{trivial} and  Theorem \ref{flat-left}, we obtain  $\K$-isomorphisms
$$
	H^i_{par}(G, Hom_{\K H} (W, U_0)) 	\simeq 	H^i(H, U_0) \simeq H^i(H, U), 
	$$
 in which  the latter isomorphism comes from the fact that $H^i(H, U_0) = Ext^i_{ \K H}(\K, U_0)$, where $\K$ is considered as a trivial right $\K H$-module, and $H^i(H, U) = Ext^i_{ \K H}(\K, U)$, with $\K$ considered as a trivial left $\K H$-module. Note that by definition $$
	H^i_{par}(G, Hom_{\K H} (W, U_0)) = Ext^i_{\K_{par} G} (B, Hom_{\K H} (W, U_0)),
	$$ where both $B$ and $Hom_{\K H} (W, U_0)$ are right $\K_{par} G$-modules.
	
	On other hand 
	$H^i_{par}(G,  V) = Ext^i_{\K_{par} G} (B,  V)$, where both $B$ and $V = W \otimes_{\K H} U$ are left $\K_{par} G$-modules. Write $V_0$ for $V$ considered as  right $\K_{par} G$-module by $v [g] = [g^{-1}] v$.  Note that $ Ext^i_{\K_{par} G} (B,  V) \simeq  Ext^i_{\K_{par} G} (B,  V_0)$, where $B$ in $Ext^i_{\K_{par} G} (B,  V_0)$ is considered as a right $\kpg$-module. Then it remains to show that \begin{equation} \label{iso-action} V_0 \simeq  Hom_{\K H} (W, U_0) \hbox{ as right } \K_{par} G\hbox{-modules}.\end{equation}

	Note that 
	\[
	V = W \otimes_{\K H} U \simeq (\K H)^{n} \otimes_{\K H} U \simeq (\K H \otimes_{\K H} U)^{n} \simeq U^{n}, 
	\]
	as $\K$-vector spaces, and composing these isomorphisms we obtain the $\K$-isomorphism 
		$$\varphi_1 : V_0 \to U^n$$
	given by
	$$
	\varphi_1 ((\lambda_1, \ldots, \lambda_n ) \otimes u )= ( \lambda_1 u, \ldots, \lambda_n u),
	$$
	where $\lambda_1, \ldots, \lambda_n \in \K H$ and $u \in U$.

	Consider also the standard isomorphism 
		$$\varphi_2 : Hom_{\K H} (W, U_0) \to U^n$$  defined by
	$$
	\varphi_2(f) = (f(e_1), \ldots, f(e_n)),
	$$
	where $\{e_1, \ldots, e_{n} \}$  is the canonical basis of the free right $\K H$-module $W \simeq (\K H)^{n} $.

	 Recall that the right action of $[g]$ on $f \in Hom_{\K H} (W, U_0)$ is determined by  $$(f \cdot  [g])(w) = f([g] \cdot w).$$ 
	 Then via the isomorphism $\varphi_2$  we can consider $U^n$ as a right $\K_{par} G$-module with $[g]$ acting in the following way :
		 if $\underline{u} = (u_1, \ldots, u_n) \in U^n$ then $ \underline{u} . [g] = \underline{\tilde{u}}$, where $\underline{\tilde{u}} = (\tilde{u}_1, \ldots, \tilde{u}_n) \in U^n$, $$M_g^*. \underline{u}^t = \underline{\tilde{u}}^t$$ 
		  and $M_{g},  M_g^*\in M_n(\K H) $ are the matrices defined in \eqref{elementaryRepr}.
		  Indeed to see that the action is precisely this one, we consider
		 $$f(e_i) = u_i, (f \cdot [g]) (e_i) = \tilde{u}_i \hbox{ and } M_g  = (\lambda_{i,j}) \in M_n(\K H).$$  
		 Then
		 $$
		 \tilde{u}_i = f([g] e_i) = f( \sum_{1 \leq j \leq n} e_j \lambda_{j,i} )  =  \sum_{1 \leq j \leq n} f(e_j) \lambda_{j,i} = \sum_{1 \leq j \leq n} u_j \lambda_{j,i} = \sum_{1 \leq j \leq n} \lambda_{j,i}^{ *} u_j .$$
	 
	  On the other hand,  for $\lambda \in (W \otimes_{\K H} U)_0 = V_0$ we have $\lambda \cdot [g] = [g^{-1}] \cdot \lambda$, so the action of $[g]$  on $U^n$    induced by $\varphi_1$ is  given as follows: if $\underline{u} = (u_1, \ldots, u_n) \in U^n$, then $ \underline{u} \cdot [g] = \underline{\tilde{u}}$, where 
	  $$M_{g^{-1}} \cdot \underline{u}^t = \underline{\tilde{u}}^t,$$ 
	  i.e., it is given by left multiplication  by the matrix $M_{g^{-1}}$, where the $n$-tuples are considered as columns.
	  
	  Finally,  $M_g^* = M_{g^{-1}}$ by Lemma \ref{action}, thus proving  
	 (\ref{iso-action}).
 
 \end{proof}

\begin{cor} \label{cohomology123}  Let   $\Delta$  be a connected component of the groupoid $\Gamma(G)$ such that $\Delta$ has finitely many vertices and  $V = W \otimes_{\K H} U$ be an irreducible (resp. indecomposable) left $\K_{par} G$-module with decomposition given by Theorem \ref{decomposition}. Then 
	$$
	H^i_{par}(G, V) \simeq H^i(H, U), \;\;\; i \geq 0.
	$$
\end{cor}

\begin{cor} {\bf (Corollary B)}
	Suppose $G$ is a finite group. Then
		$$
	H_i^{par} (G, B) \simeq \oplus_{\Delta} H_i(H, \K) \ \hbox{ and } \  H^i_{par} (G, B) \simeq \oplus_{\Delta} H^i(H, \K),
	\;\;\;  i \geq 0,
	$$
	where the sum is over the connected components $\Delta$ of the groupoid $\Gamma(G)$ and $H$ is the stabilizer of a chosen vertex of $\Delta$ (thus is defined up to conjugation).
\end{cor}

\begin{proof}
	We will split $B$ as a direct sum of indecomposable left $\K_{par} G$-modules.
	Write $$e_A = \prod_{g \in A} e_g \prod_{s \in G \setminus A} ( 1 - e_s) \in B,$$ where $A$ is a subset of $G$. The action of $[g]$ on $e_A$ gives $ [g] e_A [g^{-1}] = e_{g A}$ if $ 1 \in A$. Note that $e_A = 0$ if $1 \notin A,$  and thus $e_{gA}=0$ if $g\m \notin A.$
	Observe that  there is a $\K$-isomorphism
	$$
	B \simeq \oplus_{1 \in A} \K e_A.
	$$
	Indeed, let $G \setminus \{ 1 \} = \{ t_1, \ldots, t_k\}$. Then $$B = B e_{t_1} \oplus B ( 1 - e_{t_1}) = (B e_{t_2} \oplus B ( 1 - e_{t_2})) e_{t_1} \oplus (B e_{t_2} \oplus B ( 1 - e_{t_2})) ( 1 - e_{t_1}).$$ Continuing the decomposition of $B$ as $Be_{t_3} \oplus B ( 1 - e_{t_3})$ and so on, we obtain 
	$$
	B = \oplus_{1 \in A} B e_A.
	$$
	Note that $B = \K [e_{t_1}, \ldots, e_{t_k}]$, and  for $g \in A$ we have $e_g e_A = e_A$ and for $g \not\in A$ we have $e_g e_A = 0$. Thus $B e_A = \K  e_A$.

 Furthermore, 
for each subset $A$ with $1\in A$ the sum 
$$ \oplus _{g\in G, g^{-1}\in A} \K  e_{gA}$$ is a  left $\kpg$-module. Notice that in  the latter sum the $gA$ are the vertices of a connected component $\Delta $ of  the groupoid $ \Gamma(G).$
	
	Thus $B$ splits as a direct sum of left $\kpg$-modules  over the connected components $\Delta$ of $ \Gamma(G),$ where each summand is isomorphic to $\K \V_{\Delta}$. Then consider the epimorphisms
	$$\lambda_{\Delta} : \kpg \to \K \Delta \simeq M_n(\K H)$$
	and
	$$\theta:  M_n(\K H) \to M_n(\K),$$ 
	where $\theta$ is induced by the  epimorphism $\K H \to \K$ that sends each element of $H$ to 1. Then $\K \V_{\Delta} \simeq \K ^n = M_{n,1}(\K)$, where $M_{n,1}(\K)$ are all $n \times 1$-matrices over $\K$,  via the isomorphism that sends $e_{g_iA}$ to $e_i$. Thus $M_{n,1}(\K)$ is a left $M_n(\K)$-module via matrix multiplication,  and via the composition map $\theta \circ \lambda_{\Delta}$ it is a left $\kpg$-module. Thus 
	$$\K \V_{\Delta} \simeq W \otimes_{\K H} \K,$$
	where $W$ is the elementary bimodule related to the connected component $\Delta$.
	 Note that the above isomorphism is of left $\kpg$-modules  because $\theta(M_g) e_i = e_j$ for $g_j^{-1} g g_i = h_i \in H$ since $M_g=  \sum _{g_iA\ni g\m} E_{ji}(h_i)$.
	Then by Theorem \ref{general1} and Theorem \ref{general2}
	$$
	H^i_{par} (G, B) = \oplus_{\Delta} 	H_{par}^{i} (G, W \otimes_{\K H} \K )  \simeq \oplus_{\Delta} H^i(H, \K)$$ and $$   H_i^{par} (G, B) = \oplus_{\Delta} 	H_i^{par} (G, W \otimes_{\K H} \K )  \simeq \oplus_{\Delta} H_i(H, \K).
	$$
\end{proof}

\begin{theorem}
	Let $G$ be a finite group. Then for $i \geq 1$
	$$H^i_{par}(G, \kpg) = 0.$$
\end{theorem}

\begin{proof}
	Since $G$ is finite we have $$\kpg\simeq \K \Gamma(G) = \oplus_{\Delta} \K \Delta ,$$
	where the direct sum is over the connected components $\Delta$ of the groupoid $\Gamma(G)$. Thus $\K \Delta$ is a projective left $\kpg$-module and we  can decompose $\K \Delta \simeq W_1 \oplus \ldots \oplus W_n$, where $W_1 \simeq \ldots \simeq W_n \simeq W$ as $\kpg$-$\K H$-bimodules. Then 
 $$\K \Delta \simeq  \oplus_{ 1 \leq j \leq n} W_j \simeq \oplus_{1 \leq j \leq n} W_j \otimes_{\K H} U_j$$
	where each $U_j \simeq \K H$. Thus 
		$$	H^i_{par}(G, \kpg) \simeq \oplus_{\Delta} 	H^i_{par}(G, \K \Delta) \simeq \oplus_{\Delta}  \oplus_{1 \leq j \leq n= n(\Delta)} 	H^i_{par}(G, W_j \otimes_{\K H} \K H)$$ and by Theorem \ref{general2} for $i \geq 1$
	$$	H^i_{par}(G, W_j \otimes_{\K H} \K H) \simeq  H^i(H, \K H) = 0,$$
	where the last equality holds for any finite group $H$. Indeed  by  \cite[Prop. III.6.2]{Brown}  for any group $T$ with a subgroup of finite index $T_0$ there is an isomorphism $H^i(T, \K T) \simeq H^i(T_0, \K T_0),$ observing that for subgroups of finite index induced and co-induced modules are isomorphic   (see \cite[Prop. III.5.9]{Brown}).  Thus we obtain our result by  applying this fact for $T = H$ and  the trivial subgroup $T_0$.	
\end{proof}

\section{Proof of Theorem \ref{isomorphism.Bdelta}.} \label{section.lemma7.7}

 In this section whenever we make calculations in $\K \Delta$ all sums are over subsets of $G$ that are vertices in the connected component $\Delta$, that is, belong to $\V_{\Delta}$.

\subsection{ The morphisms $\varphi$ and $\psi$  connecting $B \otimes_{\kpg} \kd$ and $ \ld(B)$. }

\ \ 

We will introduce now an equivalence relation on $G$ which will be essential in what follows. 
Given $g \in G$, let 
\begin{equation*}
\ee_g  = \ld (g) \ld (g^{-1})  = \ld(e_g) \in \K \Delta.
\end{equation*}
We define an equivalence relation in $G$ by 
\begin{equation}
g \sim t \iff \ee_g = \ee_t .
\end{equation}
Fix a complete set of representatives $X \subset G$ for the equivalence relation defined above. 
We claim that $X$ is finite. In fact, $\Delta$ has a finite number of vertices and, by equality \eqref{formula.lambda-delta(e_g)}, for every $t \in G$ we have the formula
\begin{equation} \label{equation.epsilon}
\ee_t = \ld(e_t) =  \sum_{D \ni t}(D,1),
\end{equation}
which shows that the number of distinct $\ee_t$'s is finite.

\begin{rmk} \label{obs.1} 
	Note that  $\ee_t = \ee_g$ if and only if 	for all $C \in \V_\Delta$, 
	\begin{center}
		$C \ni g \iff  C \ni t$. 
	\end{center}
	In fact, by the definition of the equivalence relation, if $\ee_t = \ee_g$ then $\sum_{D \ni t} (D,1) = \sum_{C \ni g} (C,1)$ and vice versa. Clearly the second equality holds if and only, for every vertex $C$ of $\Delta$, $C \ni g$ implies that $C \ni t$ and conversely.
\end{rmk}	

Let $C$ be a vertex of $\Delta$. We define a partition  $X = C' \cup C''$ of the system of representatives in the following manner: 
\begin{eqnarray}
C' & = &  X \cap C \\
C''& = & X \setminus C'.
\end{eqnarray}
Note that if $x \in X$ is equivalent to some $t \in C$ then, by Remark  \ref{obs.1}, $x \in C$. 
Thus we obtain the following  description for $C'$:
\[
C' = \{ x \in X \mid  \exists t \in C \text{ such that }x \sim t \}. 
\]

\begin{rmk} \label{obs.2} 
	$C'$ determines $C$,  that is, 
	if $D, C \in \V_\Delta$ then $D'$ and $ C'$   coincide if and only if $C = D$. In fact, suppose that  $C'= D'$ and assume that there exists $d \in D \setminus C$. If $d \sim x \in X$ then $x \in D' = C' $ and hence $x \in C$. But  $d \sim x$ and $x \in C$ then, by the previous remark, $d$ also lies in $C$, a contradiction. 
\end{rmk}

Consider the $\K$-linear map
\[
\varphi: B  \otimes_{\K_{par}G} \K \Delta  \to \ld(B) , \ \ b \otimes a \mapsto \ld(b) \ract a,
\]
 where the right $\K\Delta$-action $\ract$ was defined in (\ref{right.kDelta.structure}).
This map is well-defined, since for $b \in B$, $a \in  \K \Delta$ and $[g] \in \kpg$,
\begin{align*}
\varphi ((b^{[g]}) \otimes a )& =  \varphi ( [g^{-1}] b [g] \otimes a ) \\
& = \ld (g^{-1} ) \ld( b)  \ld (g) \ract a \\
& = \ld (g)^{*} \ld( b)  \ld (g)  \ract a \\
& = (\ld( b)  \ract \ld (g ) ) \ract a \\
& = \ld( b)  \ract (\ld (g)   a) \\
& = \varphi (b \otimes \ld(g) a )\\
& = \varphi (b \otimes [g] \cdot  a ).
\end{align*}
Note that $\varphi$ is a map of right $\K \Delta$-modules: for $b \in B$ and $a_1, a_2 \in \kd$,  
\[
\varphi ( ( b \otimes a_1 ) a_2 ) = \varphi ( b \otimes a_1 a_2 ) = \ld(b) \ract (a_1 a_2) = (\ld(b) \ract a_1 ) \ract a_2 = \varphi (b \otimes a_1 ) \ract a_2. 
\]
We are going to prove  that $\varphi$ is an isomorphism of right $\K \Delta$-modules. We begin by defining a right inverse to $\varphi$, which will later be shown to be a left inverse as well. 
Given $(C,1) \in \ld(B)$, let 
\begin{equation}
\pt (C,1) = \prod_{t \in C'} e_t \prod_{f \in C''} (1 - e_f) 
\end{equation}
and extend $\pt$ linearly to $\ld(B)$. 
We now define
\begin{equation}\label{definition.psi}
\psi : \ld(B) \to B \otimes_{\kpg} \K \Delta, \ \ d \mapsto \pt(d) \otimes 1_{\K \Delta}.
\end{equation}
We  will show  that 
\begin{equation} \label{varphi.psi}
\varphi \circ \psi  = Id_{\ld(B)}.
\end{equation}
In fact, we first remark that for any finite  subset $Y$  of $G$ we have the equality 
\begin{equation} \label{formula.produto.e_t}
\prod_{t \in Y} \sum_{ \substack{D \in \V_\Delta \\ D \ni t}} (D,1) = \sum_{D \supseteq Y} (D,1), 
\end{equation}
since the elements $(D,1)$ are orthogonal idempotents in $\K \Delta$. Another simple fact needed in the  computation below is that 
if $C$ and $F$ are subsets of $G$ then 
\begin{center}
	$ F \cap C''= \varnothing$ if and only if  $F' \subseteq C'$.
\end{center}

\begin{lemma}\label{lemma.p17}
	$\ld \circ \pt = Id_{\ld(B)}$. 
\end{lemma}

\begin{proof}
	\begin{align*}
	(\ld \circ \pt)(C,1) & =  \ld (\prod_{t \in C'} e_t \prod_{f \in C''} (1 - e_f)  ) \\
	& =  \prod_{t \in C'} \ld(e_t) \prod_{f \in C''} (1 - \ld(e_f) )\\
	& =  \prod_{t \in C'} \sum_{D \ni t} (D,1)  \prod_{f \in C''} (1 - \sum_{E \ni f} (E,1))\\
	& =  \prod_{t \in C'} \sum_{D \ni t} (D,1)  \prod_{f \in C''}  \sum_{F \notni f} (F,1) \\
	& =  \left( \sum_{D \supseteq C' } (D,1) \right)   \left( \sum_{ F \cap C'' = \varnothing  } (F,1) \right)\\
	& =    \sum_{ \substack{D \supseteq C' \\  D \cap C'' = \varnothing }} (D,1). 
	\end{align*}

	Now consider a subset $D$ such that $D \supseteq C' $ and $ D \cap C'' = \varnothing$.
	The inclusion  $D \supseteq C' $ implies  $D' \supseteq C' ,$ and  $D \cap C''= \varnothing$
	yields   $D' \subseteq C' .$ Hence  $D' = C'$ and consequently $D=C,$  thanks to Remark~\ref{obs.2}. 
	Therefore 
	\[
	(\ld \circ \pt)(C,1) =   \sum_{ \substack{D \supseteq C' \\  D \cap C'' = \varnothing }} (D,1) = (C,1).
	\] \end{proof}  

The proof of the above lemma immediately  implies the following result, which is precisely (\ref{varphi.psi}).

\begin{cor}\label{phipsiC,1}
	Given $(C,1)$ in $\K \Delta$,  we have that 
	$$\varphi \circ \psi (C,1)  = (C,1).$$
\end{cor}

\subsection{Linearity of $\psi$}

\ \

We begin with  two technical results which will be useful in the proof that $\psi$ is a $\K \Delta $-linear map, that is, $\psi$ is a homomorphism of right $\K \Delta$-modules.

\begin{rmk} For all $e,f$ in $B$, $f^{ e }= fe$. \\
	In fact, letting $e = e_g$  and using that $B$ is a commutative ring, we have:
	\[
	f ^{e_g} = f   ^{[g][g^{-1}]} = [g] ([g^{-1}] f [g] )[g^{-1}] = e_g f  e_g = f e_g.
	\]
\end{rmk}

\begin{lemma} \label{lemma.october}
	For all $t_1, t_2$ in $G$, if $t_1 \sim t_2$ then 
	
	\begin{enumerate}
		\item[(i)] $
		e_{t_1} \otimes 1_{\K \Delta} = e_{t_2} \otimes 1_{\K \Delta} = e_{t_1}e_{t_2} \otimes 1_{\K \Delta} \in B \otimes_{\kpg }\kd.
		$
		\item[(ii)] 
		$
		[g^{-1}] e_{t_1} [g] \otimes 1_{\K \Delta} =  [g^{-1}] e_{t_2} [g] \otimes 1_{\K \Delta} = [g^{-1}] e_{t_1}e_{t_2} [g] \otimes 1_{\K \Delta}, 
		$ for all $g \in G$. 		
	\end{enumerate}	
	
\end{lemma}
\begin{proof}
	(i) \begin{align*}
	e_{t_1} \otimes 1_{\K \Delta} & = e_{t_1}^2 \otimes 1_{\K \Delta} = e_{t_1} ^{e_{t_1}} \otimes 1_{\K \Delta} 
	= e_{t_1} \otimes \ld(e_{t_1}) =  e_{t_1} \otimes \ld(e_{t_2}) \\ 
	& = e_{t_1} ^{e_{t_2}} \otimes 1_{\K \Delta} = e_{t_1}e_{t_2} \otimes 1_{\K \Delta} . 
	\end{align*}
	Switching $t_1$ and $t_2$ we get $e_{t_2} \otimes 1_{\K \Delta} =  e_{t_2}e_{t_1} \otimes 1_{\K \Delta} = e_{t_1}e_{t_2} \otimes 1_{\K \Delta}$.
	
	(ii) Given $g \in G$, \begin{align*}
	[g^{-1}] e_{t_1} [g] \otimes 1_{\K \Delta} & = (e_{t_1} ^{[g]}) \otimes 1_{\K \Delta} 
	= (e_{t_1} ^{e_{t_1}}) ^{[g]} \otimes 1_{\K \Delta} 
	=  e_{t_1} ^{e_{t_1}  [g]} \otimes 1_{\K \Delta} \\
	& =  e_{t_1} \otimes \ld( e_{t_1}  [g])  
	=  e_{t_1} \otimes \ld( e_{t_1}) \ld ( [g])  
	=  e_{t_1} \otimes \ld( e_{t_2}) \ld ( [g])  \\
	&  =  e_{t_1} \otimes \ld( e_{t_2} [g])  
	= (e_{t_1} ^{e_{t_2} [g]}) \otimes 1_{\K \Delta}
	= (e_{t_1} ^{ e_{t_2}} ) ^{[g]} \otimes 1_{\K \Delta} \\
	& = (e_{t_1} e_{t_2} )^{[g]} \otimes 1_{\K \Delta}
	= [g^{-1}] e_{t_1} e_{t_2}  [g] \otimes 1_{\K \Delta}. 
	\end{align*}
	Clearly we also have 
	\[
	[g^{-1}] e_{t_2} [g] \otimes 1_{\K \Delta}  
	= [g^{-1}] e_{t_2} e_{t_1}  [g] \otimes 1_{\K \Delta} 
	= [g^{-1}] e_{t_1} e_{t_2}  [g] \otimes 1_{\K \Delta}, 
	\]
	concluding the proof.  
\end{proof}

In what follows,  given a property $P$  the symbol  $[ P ]$  denotes the following Boolean operator: 
\[
[P] = 
\left\{
\begin{array}{ll}
1 & \text{if $P$ holds}, \\
0 & \text{otherwise}.
\end{array}
\right. 
\]
We are going to prove the following fact.
\begin{prop}\label{psi-kdelta-linear} The map  
	$\psi : \ld(B) \to B \otimes_{\kpg} \K \Delta$ defined in \eqref{definition.psi} is a homomorphism of right 
	$\K \Delta $-modules.
\end{prop}

\begin{proof} We  shall show that 
	\[
	\psi ( (C,1) \ract \ld(g))  = \psi  (C,1) \ld(g) 
	\]
	for all $(C,1) \in \ld(B)$ and $g \in G$. 
	On one hand, we have 
	\begin{align*}
	\psi ( (C,1) \ract \ld(g)) & = \pt \left((C,1) \ract \ld(g)\right) \otimes 1_{\kd} \\
	& = \pt \left( (C,1) \ract \sum_{D \ni g^{-1}} (D,g) \right) \otimes 1_{\kd} \\
	& =\sum_{D \ni g^{-1}} \pt \left(  (gD,g^{-1}) (C,1) (D,g) \right)\otimes 1_{\kd} \\
	& = [g \in C ] \pt  (g^{-1} C, 1) \otimes 1_{\kd} \\
	& = [g \in C ]  \prod_{ t \in (g^{-1} C)'}e_t \prod_{f \in (g^{-1} C) ''} (1 - e_f)  \otimes 1_{\kd}.
	\end{align*}
	
	Since $C \ni 1$ we have that $g^{-1} \in g^{-1} C$, and as a consequence $g^{-1} \sim x_{g^{-1}} \in (g^{-1}C) '$. It follows from Lemma
	\ref{lemma.october} that 
	\[
	e_{x_{g^{-1}}} \otimes \one = e_{g^{-1}} \otimes \one
	\]
	and hence
	\begin{align*}
	\psi ( (C,1) \ract \ld(g)) & = [g \in C ]  \prod_{ t \in (g^{-1} C)'}e_t \prod_{f \in (g^{-1} C) ''} (1 - e_f)  \otimes 1_{\kd} \\
	& = [g \in C ] \ e_{x_{g^{-1}}} \prod_{ t \in (g^{-1} C)'} e_t \prod_{ f \in (g^{-1} C)''} (1 - e_f)  \otimes \one \\
	& = [g \in C ] \ e_{g^{-1}} \prod_{ t \in (g^{-1} C)'} e_t \prod_{ f \in (g^{-1} C)''} (1 - e_f)  \otimes \one\\
	& = (*).
	\end{align*}
	
	On the other hand
	\begin{align*}
	\psi (C,1)  \ld(g) & = (\pt (C,1) \otimes 1_{\kd} ) \ld(g) = \pt(C,1) \otimes \ld(g) \\ 
	& = (\pt (C,1) ^{[g]}) \otimes 1_{\kd} \\
	& = \left(\prod_{t \in C'} e_t \prod_{t \in C''} (1 - e_t)    \right) ^{[g]} \otimes 1_{\kd} \\
	& = [g^{-1}] \left(\prod_{t \in C'} e_t \prod_{f \in C''} (1 - e_f)    \right) [g] \otimes 1_{\kd} \\
	& = \prod_{t \in C'} e_{g^{-1}t}  \prod_{f \in C''} (1 - e_{g^{-1}f})  e_{g^{-1}} \otimes 1_{\kd} \\
	& = e_{g^{-1}}\prod_{t \in C'} e_{g^{-1}t}  \prod_{f \in C''} (1 - e_{g^{-1}f})  \otimes 1_{\kd} \\
	& = (**).
	\end{align*}
	
	We shall  check that (*) = (**). 
	
	\begin{claim}  If $g\in C,$ then 
		
		\begin{equation} \label{first.products} 
		e_{g^{-1}} \prod_{ t \in (g^{-1} C)'} e_t \otimes \one = e_{g^{-1}}\prod_{t \in C'} e_{g^{-1}t}  \otimes 1_{\kd}.
		\end{equation} 
		
	\end{claim}
	
	Indeed, if $t \in C'$ then  $g^{-1}t   \in g^{-1}C' \subseteq g^{-1}C$ and hence $g^{-1}t \sim x_{g^{-1}t} \in (g^{-1}C)' $. 
	Using once more Lemma \ref{lemma.october}, we obtain the equality  
	\[
	e_{g^{-1}t} \otimes \one = e_{x_{g^{-1}t}} \otimes \one , 
	\] 
	which ensures that  every element $e_{g^{-1}t} $, with $ t\in C' $, occurring in the RHS of \eqref{first.products}  may be substituted by some $e_{t'} $, $ t'\in (g^{-1}C)' $, which occurs in the LHS of \eqref{first.products}.
	
	Conversely, consider an element $e_{t'} $ with $t' \in (g^{-1}C)'$ as in the LHS of \eqref{first.products}. There is $\beta \in C$ such that $t' = g^{-1} \beta$; let $x_\beta \in C'$ be its representative, that is, $\beta \sim x_\beta$. Then 
	\begin{align*}
	e_{g^{-1}}e_{t'} & = e_{g^{-1}}e_{g^{-1}\beta} =e_{g^{-1}\beta} e_{g^{-1}} = [g^{-1}]e_\beta[g]  ,
	\end{align*}
	hence, by Lemma \ref{lemma.october}, 
	\[
	e_{g^{-1}}e_{t'} \otimes \one  = [g^{-1}]e_\beta[g] \otimes \one  = 
	[g^{-1}]e_{x_\beta}[g] \otimes \one = e_{g^{-1}}e_{g^{-1}x_\beta} \otimes \one.
	\]
	It follows that every $e_{t'}$, where $t'\in (g^{-1}C)'$, occurring in the LHS of \eqref{first.products} may be replaced by an $e_{g^{-1}t}$, $t \in C'$, as in the RHS of \eqref{first.products}, which completes the proof of our claim. 
	
	This takes  care of  the first products appearing in (*) and (**). Now we will   consider the second products.
	
	\begin{claim} If $g\in C,$ then 
		\begin{equation} \label{second.products}
		e_{g^{-1}} \prod_{ f \in (g^{-1} C)''} (1 - e_f)  \otimes \one = e_{g^{-1}} \prod_{f \in C''} (1 - e_{g^{-1}f})   \otimes 1_{\kd}. 
		\end{equation}
	\end{claim}
	
	If $f \in C''$ then $f \notin C'$ and therefore $g^{-1 } f \notin g^{-1}C$. Let $x_{g^{-1}f} \in X$ be the representative of $g^{-1}f$. 
	Since  $g^{-1 } f \notin g^{-1}C$, by Remark \ref{obs.1} we also have that $x_{g^{-1 } f} \notin g^{-1}C$, which implies that
	$x_{g^{-1 } f} \in (g^{-1}C)''$.
	As a consequence of Lemma \ref{lemma.october}, 
	\[
	(1 - e_{g^{-1 } f}) e_{g^{-1}} \otimes \one = 
	(1 - e_{x_{g^{-1 } f}}) e_{g^{-1}} \otimes \one. 
	\]
	Hence each factor $1 - e_{g^{-1}f}$, where $f \in C''$, occurring in the RHS of \eqref{second.products} may be replaced by a factor $1 - e_t$, with $t \in (g^{-1}C)''$, in the LHS of  \eqref{second.products}.

	Pick a factor $1-e_f$ of the LHS of  \eqref{second.products}, where $f \in (g^{-1}C)''$. 
	\[
	f \in (g^{-1}C)'' \implies f \in X \setminus g^{-1}C  \implies f \notin g^{-1}C \implies gf \notin C. 
	\]
	Let $x_{gf} \in X$ be the representative of $gf$. 
	Since $gf \notin C$, once more Remark \ref{obs.1} implies  that $x_{gf} \notin C$ and therefore $x_{gf} \in C''$. 
	Using  again Lemma ~ \ref{lemma.october}, we have that 
	\begin{align*}
	& (1 - e_{g^{-1}x_{gf}}) e_{g^{-1}} \otimes \one  = 
	(1 - e_{g^{-1}x_{gf}}) [g^{-1}][g] \otimes \one = \\
	= &  [g^{-1}] (1 - e_{x_{gf}})[g] \otimes \one  =   
	[g^{-1}] (1 - e_{gf})[g] \otimes \one  \\ =  &  (1 - e_{f}) e_{g^{-1}} \otimes \one . 
	\end{align*}
	Therefore each term $(1-e_f)$, where $f \in (g^{-1}C)''$, occurring in the LHS of  \eqref{second.products} may be replaced by a corresponding term $(1-e_{g^{-1}x})$, where $x \in C''$, in the RHS of  \eqref{second.products},
	and the claim follows. 
	
   Claims \eqref{first.products} and \eqref{second.products} above imply that  if $g\in C,$ then 
	\[
	\ e_{g^{-1}} \prod_{ t \in (g^{-1} C)'} e_t \prod_{ f \in (g^{-1} C)''} (1 - e_f)  \otimes \one
	= e_{g^{-1}}\prod_{t \in C'} e_{g^{-1}t}  \prod_{f \in C''} (1 - e_{g^{-1}f})  \otimes 1_{\kd} 
	\] 
	and therefore the equality (*) = (**) holds 	when $g \in C$ .

	When $g \notin C$ we have (*) = 0 ; we shall see that in this case (**) = 0 as well. 
	Assuming $g \notin C$,  choose $x_g \in X$ such that $x_g \sim g$. In this case $x_g \in C''$ and, by  Lemma ~ \ref{lemma.october},
	\[
	e_{g^{-1}} \otimes \one = [g^{-1}]e_g [g] \otimes \one = 
	[g^{-1}]e_{x_g} [g] \otimes \one  = e_{g^{-1}x_g} e_{g^{-1}} \otimes \one
	\]
	and since the term $1 - e_{g^{-1}x_g}$ appears in (**)  it follows that (**)=0. 
	
	This proves the equality $ \psi ( (C,1) \ract \ld(g))  = \psi  (C,1)  \ld(g)$
	for all $(C,1) \in \ld(B)$ and $g \in G$. 
	Since by Proposition 2.2. of \cite{DdLP}  the algebra $\kd$ is generated by the elements $\ld(g)$, with $g \in G$, we conclude that $\psi$ is a map of right $\K \Delta$-modules. \end{proof}

\subsection{ The equality $\psi \circ \varphi = Id$. Conclusion of the proof of Theorem \ref{isomorphism.Bdelta}.}

\begin{prop}
	$\pt : \ld(B)\to B$ is a multiplicative map. 
\end{prop}
\begin{proof}
	\[
	\pt \left( (C,1) (D,1) \right) = [C = D] \pt (C,1) = [C = D] \prod_{t \in C'} e_t \prod_{f \in C''} (1 - e_f), 
	\]	
	and 
	\[
	\pt (C,1) \pt (D,1) =  \prod_{t_1 \in C'} e_{t_1} \prod_{f_1 \in C''} (1 - e_{f_1})
	\prod_{t_2 \in D'} e_{t_2} \prod_{f_2 \in D''} (1 - e_{f_2}) .
	\]
	
	Clearly, if $C = D$ then  $  \pt (C,1) ^2  = \pt \left( (C,1)^2 \right) $. 
	Assume that $C \neq D$. Then either $C \setminus D \neq \varnothing$ or $D \setminus C \neq \varnothing$. 
	If  there is $\alpha \in C   \setminus D$ then $\alpha \sim x_\alpha \in D''$; but if $\alpha \in C$ then $x_\alpha \in C'$ and hence 
	$\pt (C,1) \pt (D,1) = 0 = \pt \left( (C,1) (D,1) \right) $. 
	The other case is analogous, and hence $\pt$ preserves products.  \end{proof}

\begin{lemma} \label{lemma.A} 
	$\pt (\one) \otimes \one = \idB \otimes \one $.
\end{lemma}

\begin{proof}
	For each $S \subseteq X$, let $P_S \in B $ be the element
	\[
	P_S = \prod_{t \in S} e_t \prod_{f \in X \setminus S} (1 - e_f).
	\]

	\begin{claim} \label{tensor.nulo}
		Let $C \in \V_\Delta$. If $S \neq C' $ then 	$P_S \otimes (C,1) = 0$ in  $B \otimes_{\kpg} \kd$. 
	\end{claim}
	
	\begin{proof}
		Case 1. Assume that there exists $t_0 \in S \setminus C' $ (which is the same as requiring that $t_0 \in S \setminus C$). 
		\begin{align*}
		P_S \otimes (C,1) &  = \prod_{t \in S} e_t \prod_{f \in X \setminus S} (1 - e_f) \otimes (C,1) \\
		& = e_{t_0}\prod_{t \in S} e_t \prod_{f \in X \setminus S} (1 - e_f) \otimes (C,1) \\
		& = (P_S ^{e_{t_0}} ) \otimes (C,1) \\
		& = P_S \otimes \ld (e_{t_0}) (C,1)\\
		& = P_S \otimes \sum_{D \ni t_0} (D,1) (C,1) = 0, 
		\end{align*}	
		since $D \neq C$ for all such $D$.

		Case 2. Suppose that there exists $t_0 \in C' \setminus S$. Then $t_0 \in X \setminus S$, and 
		\begin{align*}
		P_S \otimes (C,1) &  = P_S (1 - e_{t_0}) \otimes (C,1) \\
		& = (P_S ^{(1 - e_{t_0})}) \otimes (C,1) \\
		& = P_S   \otimes \ld (1 - e_{t_0})  (C,1) \\
		& = P_S   \otimes \left(1 -  \sum_{D \ni t_0} (D,1)  \right)  (C,1) \\
		& = P_S \otimes \sum_{E \notni t_0} (E,1) (C,1) = 0,
		\end{align*}
		where the last equality holds because $t_0 \in C' \subseteq C$. 
	\end{proof}
	
	Note that 
	\begin{equation}
	\idB = \sum_{S \subseteq X} P_S
	\end{equation}
	since 
	\[
	\idB = \prod_{t \in X} \idB  
	=  \prod_{t \in X} ((\idB -e_t) + e_t)
	= \sum_{S \subseteq X} P_S.
	\]
	It follows that
	\begin{equation}
	\idB \otimes \one = \sum_{D \in \V_\Delta} P_{D'} \otimes \one.
	\end{equation}
	Indeed, by Claim ~ \ref{tensor.nulo}, 
	
	\begin{align*}
	\idB \otimes \one 
	& = \sum_{S \subseteq X} P_S \otimes \one 
	= \sum_{S \subseteq X} P_S \otimes \sum_{C \in \V_\Delta} (C,1) 
	= \sum_{C\in \V_\Delta} P_{C' } \otimes (C,1) \\ 
	& = 
	\sum_{D \in \V_\Delta} P_{D'} \otimes \sum_{C \in \V_\Delta} (C,1) = \sum_{D \in \V_\Delta} P_{D'} \otimes \one .
	\end{align*}
	
	Finally
	\[
	\pt (\one) \otimes \one = \pt \left(\sum_{D \in \V_\Delta} (D,1) \right) \otimes \one = \sum_{D \in \V_\Delta} P_{D'} \otimes \one = \idB \otimes \one. 
	\]
\end{proof}

\begin{lemma} \label{lemma.p19}
	$\psi \circ \varphi \mid_{B \otimes  \one} = Id_{B \otimes  \one}$.
\end{lemma}

\begin{proof}
Let  $g \in G$. 
	\begin{align*}
	& \psi \circ \varphi  (e_g \otimes 1_{\K \Delta}  ) = \psi ( \lambda_{\Delta } (e_g) \ract 1_{\K \Delta} )  =  \psi (\ld (e_g))  = 	 \tilde{\pi} \left( \ld (e_g) \right)  \otimes 1_{\K \Delta} \\ = & \tilde{\pi} \left( \sum_{g \in C} (C,1) \right) \otimes 1_{\K \Delta} 
	= \sum_{g \in C} \prod_{t \in C'} e_t \prod_{f \in C''} (1 - e_f) \otimes 1_{\K \Delta}. \\
	\end{align*}

	Let $C$ be any vertex of $\Delta$ containing $g$ and let $x_g \in C' $ be such that $x_g \sim g$. From Lemma \ref{lemma.october} it follows that 
	$
	e_{x_g} \otimes 1_{\K\Delta}  = e_{x_g}e_g \otimes 1_{\K\Delta}
	$
	and therefore
	\[
	\psi \circ \varphi (e_g \otimes 1_{\K\Delta})  =   e_g  \sum_{g \in C} \prod_{t \in C'} e_t \prod_{f \in C''} (1 - e_f) \otimes 1_{\K \Delta}
	= e_g ( \psi \circ \varphi )  (e_g \otimes 1_{\K \Delta}  ).
	\]
	
	Let $D \in \V_\Delta$ be such that $g \notin D$. Then $g \sim x_g \in D''$ and it follows from Lemma \ref{lemma.october} that 
	$(1 - e_{x_g} )\otimes 1_{\K \Delta} = (1 - e_{g} )\otimes 1_{\K \Delta}$. 
	
	Hence, 
		
	\begin{align*}
	& (\psi \circ \varphi) (e_g \otimes \one )  = e_g \sum_{C \ni g} \prod_{t \in C'} e_t \prod_{f \in C'' } (1 - e_f) \otimes \one = \\
	& =  e_g \left(  \sum_{C \ni g} \prod_{t_1 \in C'} e_{t_1} \prod_{f_1 \in C'' } (1 - e_{f_1}) 
	+ (1 - e_g) \sum_{D \notni  g} \prod_{t_2 \in D'} e_{t_2} \prod_{f_2 \in D'' } (1 - e_{f_2})\right)  \otimes \one \\
	& =  e_g \left(  \sum_{C \ni g} \prod_{t_1 \in C'} e_{t_1} \prod_{f_1 \in C'' } (1 - e_{f_1}) 
	+  \sum_{D \notni  g} \prod_{t_2 \in D'} e_{t_2} \prod_{f_2 \in D'' } (1 - e_{x_g})(1 - e_{f_2})\right)  \otimes \one 
	\\
	& =  e_g \left(  \sum_{C \ni g} \prod_{t_1 \in C'} e_{t_1} \prod_{f_1 \in C'' } (1 - e_{f_1}) 
	+  \sum_{D \notni  g} \prod_{t_2 \in D'} e_{t_2} \prod_{f_2 \in D'' } (1 - e_{f_2})\right)  \otimes \one 
	\end{align*}
	
	\begin{align*} 
	&   =  e_g \left(  \sum_{C \ni g} \pt (C,1) +  \sum_{D \notni  g} \pt(D,1) \right) \otimes \one 
	=  e_g  \left(  \sum_{C  \in \V_{\Delta}}  \pt (C,1) \right) \otimes \one \\
	& =  e_g \pt (\one) \otimes \one \overset{Lemma ~ \ref{lemma.A}}{=}  e_g \idB \otimes \one  = e_g \otimes \one. 
	\end{align*}	
\end{proof}

To complete the proof of Theorem \ref{isomorphism.Bdelta}, we show that 
$\varphi : B \otimes_{\kparg } \K \Delta \to \ld(B)$ is an isomorphism of right $\K\Delta$-modules with inverse $\varphi^{-1} = \psi$.

Indeed, given $b \otimes v \in B \otimes_{\kparg } \K\Delta $, 
\begin{align*}
(\psi \circ \varphi) (b \otimes v) 
& = \psi (\varphi ( b \otimes \one ) v )   = ((\psi \circ \varphi) (b \otimes \one)) v 
\end{align*}
by Proposition \ref{psi-kdelta-linear}, and  Lemma \ref{lemma.p19} yields  
$$ ((\psi \circ \varphi) (b \otimes \one)) v  = (b \otimes \one ) v = b \otimes v.  $$

On the other hand, given $(D,1) \in \ld(B)$, by Corollary \ref{phipsiC,1} we have that 
\[
(\varphi \circ \psi) (D,1) = (D,1), 
\]
and Theorem \ref{isomorphism.Bdelta} is proved.

 \section{Projective and flat $\kpg$-modules}

 In this section, whenever we make calculations in $\K \Delta$, all sums are over subsets of $G$ that are vertices in the connected component $\Delta$, i.e., subsets that belong to $\V_{\Delta}$.

 When $G$ is a finite group, the partial group algebra $\kpg$ is isomorphic to the algebra of the groupoid $\Gamma(G)$. This finite groupoid is the disjoint union of its connected components $\Delta$, and therefore   
$\kpg \simeq \K \Gamma(G) = \oplus_{\Delta} \K \Delta$ as an algebra, where the sum runs over the components $\Delta$ of $\Gamma(G)$, and each $\K \Delta$ is a two-sided ideal of $\kpg$ which is generated by a central idempotent. 
It then follows that each  $\K \Delta$ is a projective $\kpg$-module when $G$ is a finite group.

The isomorphism $\kpg \simeq \K \Gamma(G) $ does not hold anymore for an infinite group, and we cannot apply the same reasoning for the $\kpg$-modules $\K \Delta$. In fact, as we show in Section \ref{projective-section}, when $G$ is infinite there are certain components $\Delta$ such that $\K \Delta$ is not a  projective $\kpg$-module.  
 However, we also show in 
Section \ref{flat-section} that each algebra $\K \Delta$ is still a \textit{flat} $\kpg$-module.

 \subsection{Non-projective $\kpg$-modules} \label{projective-section}

\begin{prop} \label{proj99}
Suppose that $\Delta$ is a connected component of the groupoid $\Gamma(G)$ with finitely many vertices, $A$ is a vertex of $\Delta$ and $G \setminus A$ is infinite. Then $\K \Delta$ is not a projective left 
$\K_{par} G$-module.
        \end{prop}
\begin{proof}
Recall that we view $\K \Delta$ as $\kpg$-module via the map $\lambda_{\Delta} : \kpg \to \K \Delta$. Note that by \cite[Prop. 2.2]{DdLP}, $\K \Delta$ is generated as a $\K$-algebra by $\{ \lambda_{\Delta}([g]) \ | \ g \in G \}$ and since $\lambda_{\Delta}$ is a homomorphism of $\K$-algebras we deduce that $\lambda_{\Delta}$ is surjective.

 Suppose that $\K \Delta$ is projective as a $\kpg$-module.
Then there is a homomorphism of $\kpg$-modules $f : \K \Delta \to \kpg$ such that $\lambda_{\Delta} \circ f = id_{\K \Delta} .$
 Then for $g^{-1} \in A$ we have
$$ [g] f( (A,1)) = f( \lambda_{\Delta}( [g]) (A,1) ) = f( \sum_{g^{-1} \in C, C \in \V_{\Delta}} (C,g) (A,1)) = f( ( A,g)),$$
and for $h \in G, g^{-1} h^{-1} \not\in A, g^{-1} \in A$ we have
$$[h][g]f((A,1)) = [h] f((A,g)) = f( \lambda_{\Delta}([h]) (A,g)) =  f( \sum_{h^{-1} \in C, C \in \V_{\Delta}} (C,h) (A,g)) = f(0) = 0.$$
Suppose $f((A,1)) =  \sum_{s \in S} [s] \mu_s$, where $ \mu_s \in B$. Then for $g =1_G$, $h^{-1} \in G \setminus A$ we have
$$0 = [h]f((A,1)) = \sum_{s \in S} [h][s] \mu_s = \sum_{s \in S} [h s] e_{s^{ -1}} \mu_s,$$
hence $[h s] e_{s^{ -1}} \mu_s  = 0$ for every $s \in S$. Thus $e_{s^{ -1}} \mu_s  \in \{ b \in B \ | \ [hs]b = 0 \} = (1 - e_{(hs)^{-1}}) B$, hence
$$ e_{s^{ -1}} \mu_s \in \bigcap_{ h^{-1} \in G \setminus A} ( 1 - e_{(hs)^{-1}}) B = 0.$$
The last equality comes from the fact that $G \setminus A$ is infinite. 
Then $[s] \mu_s = [s] e_{s^{-1}} \mu_s = [s] \cdot 0 = 0$ and for  $g_0^{-1} \in A$, $$f((A,g_0)) = [g_0] f((A,1)) =  [g_0] \sum_{s \in S} [s] \mu_s = 0.$$ The same argument shows that for $C \in \V_{\Delta}$, $g^{-1} \in C$ we have  $f((C,g)) = 0,$  as  $G \setminus C$ is also infinite.  Therefore, $f$ is the zero map.
Finally $id_{\K \Delta} = \lambda_{\Delta} \circ f = 0$, a contradiction.
\end{proof}

\begin{prop} \label{not-projective} Suppose  that $G$ is an infinite group  and consider the connected component $\Delta$ with unique vertex $G$. Then $\K \Delta$ is not projective as a left $\kpg$-module.
\end{prop}

\begin{proof}
	
	\medskip First note that $\K\Delta \simeq M_{1} (\K G) \simeq \K G$; henceforth we will identify the algebras $\K\Delta$ and $\K G$.
	Suppose  that $\K G$ is projective as a left $\kpg$-module.  Then there is a homomorphism of left $\kpg$-modules $$\theta : \K G \to \kpg \hbox{ such that  }\lambda_{\Delta} \circ \theta = id_{\K G}.$$ Note that
	$$\lambda_{\Delta}([g]) = g \hbox{ and } \lambda_{\Delta}(e_g) = 1.$$
	Then
	$$
	\theta(g) = ( 1 + w_{g,g})[g] + \sum_{g_1 \in G \setminus \{ g \}} w_{g, g_1} [g_1],
	$$
	where each $w_{g, g_1} \in \Omega = \Ker (\lambda_{\Delta}) \cap B$. Then since $\theta$ is a homomorphism of left $\kpg$-modules
	 $$ [t] \theta(g) = \theta(tg) \hbox{ for } t,g \in G,$$
	hence
	\begin{equation} \label{calculation123}
	[t]( 1 + w_{g,g})[g] + \sum_{g_1 \in G \setminus \{ g \}} [t] w_{g, g_1} [g_1] = ( 1 + w_{tg,tg})[tg] + \sum_{g_2 \in G \setminus \{ tg \}} w_{tg, g_2} [g_2].\end{equation}
	This together with the rules of multiplication in $\kpg = \oplus_{g \in G} B[g]$ : 
	$$
	[g] e_h = e_{gh} [g] \hbox{ and }[g][h] = e_g [gh] \hbox{ for } g,h \in G,
	$$
	imply that 
	$$	[t]( 1 + w_{g,g})[g] = ( 1 + w_{tg,tg})[tg] \hbox{ and }  [t] w_{g, g_1} [g_1] = w_{tg, tg_1} [t g_1]  \hbox{ for } g,t \in G, g_1 \in G \setminus \{ g \}.
	$$
	Note that
	$[t](1 + w_{g,g}) \in [t] B = B[t]$, hence 
	$ ( 1 + w_{tg,tg})[tg]  = 	[t]( 1 + w_{g,g})[g] \in B [t][g] = B e_t [tg]$.  Therefore, 
$$  1 + w_{tg,tg} \in  B e_t + B(1-e_{tg})  \hbox{ for every } t,g \in G.$$
	Fix $tg = g_0$, thus 
	\begin{equation} \label{uni1}
	 1 + w_{g_0,g_0} \in \cap_{h \in G} (B e_{h}  + B(1-e_{g_0})). \end{equation}
	\begin{claim} \label{claim.null.intersection}
	 If $G$ is infinite we have $\cap_{h \in G} (B e_h + B(1-e_{g_0}))= B( 1 - e_{g_0})$.
	\end{claim}
	
	\medskip Note that the Claim together with (\ref{uni1})  implies that $ 1 + w_{g_0, g_0} \in B ( 1 - e_{g_0}) \subseteq \Omega$, that together with  	 $w_{g_0,g_0} \in \Omega$ gives $ 1 \in \Omega$, a contradiction.
	 
	 \medskip{\bf Proof of Claim \ref{claim.null.intersection}}
	 Note that $\K_{par}G$ is isomorphic to $\K S(G)$, where $S(G)$ is Exel's inverse semigroup. Every element of $S(G)$ has a canonical form $e_{g_1} \ldots e_{g_s} [g]$, that is unique up to permutation of the factors  $e_{g_1}, \ldots, e_{g_s}$. Thus the elements of $B$  are  $\K$-linear combinations of  $e_{g_1} \ldots e_{g_s}$ and $e_1 = 1$, where $g_1, \ldots, g_s \in G \setminus \{ 1 \}$ are pairwise distinct, and  this decomposition is unique (up to permutation of the factors in each product). Thus one such $\K$-linear combination is in   $Be_h  e_{g_0}$ for $h \not= g_0$ precisely when two of  $e_{g_1}, \ldots, e_{g_s}$ are $e_{g_0}$ and $e_h$  for every summand of the $\K$-linear combination. This together with the fact that $G$ is infinite implies that  $\cap_{h \in G} B e_h e_{g_0} = 0$.
	 
	 Suppose that $\lambda \in \cap_{h \in G} (B e_h+ B(1-e_{g_0}))$. Then
	 	$\lambda e_{g_0} \in \cap_{h \in G} B e_h e_{g_0} = 0$, hence $\lambda \in B( 1 - e_{g_0})$.
\end{proof}

 \subsection{Flat $\kpg$-modules} \label{flat-section}

\begin{lemma} \label{lemma-flat}
	Let $T$ be a subset of $G \setminus \{ 1 \}	
	$ and let $B_0$ be the unital $\K$-subalgebra of $B$ generated by $\{ e_g ~| ~ g \in T \}$. 
	
	a) Consider $\K$ as $B_0$-module, where each $e_g$ acts on $\K$ as 1. Then $\K$ is flat as $B_0$-module.
	
	b) Consider $\K$ as $B_0$-module, where each $e_g$ acts on $\K$ as 0. Then $\K$ is flat as $B_0$-module.
\end{lemma}

\begin{proof}
	Note that there is a  $\K$-automorphism of $B_0$ that sends $e_g$ to $ 1 - e_g$ for each $g \in T$.  Note that in case b) the idempotent $1 - e_g$ acts as 1, thus we can apply case a). 
	
	Now consider case a).  We will show that 
	$$Tor_1^{B_0}(-, \K) = 0,$$
	which is equivalent to the fact that  $\K$ is flat as a left $B_0$-module.\\ 
	Let $g_1, \ldots, g_k$ be pair-wise different elements of $T$. Then by induction on $k$ we have
	\begin{equation} \label{soma-direta}
	B_0 = (\oplus_{0 \leq i \leq k-1} e_{g_1} \ldots e_{g_i} (1 - e_{g_{i+1}})B_0)  \oplus e_{g_1} \ldots e_{g_k} B_0,
	\end{equation}
	where for $i = 0$ the element $e_{g_1} \ldots e_{g_i}$ is 1.
	Set
	 $$M_{ \{g_1, \ldots, g_k \}} = \sum_{1 \leq i \leq k} (1 - e_{g_i}) B_0 \hbox{ and } W_{ \{g_1, \ldots, g_k \}} = \oplus_{0 \leq i \leq k-1} e_{g_1} \ldots e_{g_i} (1 - e_{g_{i+1}})B_0.$$ 	
	We claim that $M_{ \{g_1, \ldots, g_k \}} = W_{ \{g_1, \ldots, g_k \}}$, i.e. \begin{equation} \label{uni2} M_{ \{g_1, \ldots, g_k \}} = \oplus_{0 \leq i \leq k-1} e_{g_1} \ldots e_{g_i} (1 - e_{g_{i+1}})B_0.\end{equation}
	Indeed, $W_{ \{g_1, \ldots, g_k \}} \subseteq M_{ \{g_1, \ldots, g_k \}}$, so by (\ref{soma-direta}),
	$$
	M_{ \{g_1, \ldots, g_k \}} = W_{ \{g_1, \ldots, g_k \}} \oplus ( M_{ \{g_1, \ldots, g_k \}} \cap e_{g_1} \ldots e_{g_k} B_0).$$
	Note that multiplication with $e_{g_1} \ldots e_{g_k}$ acts on $e_{g_1} \ldots e_{g_k} B_0$ as the identity map and acts on $ M_{ \{g_1, \ldots, g_k \}}$ as the zero map, hence $ M_{ \{g_1, \ldots, g_k \}} \cap e_{g_1} \ldots e_{g_k} B_0 = 0$ and (\ref{uni2}) holds.
By (\ref{soma-direta}) and (\ref{uni2}) $M_{ \{g_1, \ldots, g_k \}}$	
	is a $B_0$-module direct summand of $B_0$, 
	hence $M_{ \{g_1, \ldots, g_k \}}$  is a projective $B_0$-module, and thus it is a flat $B_0$-module. 
	
	Let $\Omega$ be the ideal of $B_0$ generated by $\{ e_g - 1 \mid g \in T \}$  i.e. the $B_0$-submodule generated by $\{ e_g - 1 \mid g \in T \}$. Note that $\Omega$ is  the union of 
	the submodules $M_{ \{g_1, \ldots, g_k \}}$, where $\{ g_1, \ldots, g_k \} \subseteq T$. Thus $\Omega$ is the direct limit of $M_{ \{g_1, \ldots, g_k \}}$, where the direct system is defined by the inclusion maps $M_{N_1} \to M_{N_2}$ for finite subsets $N_1 \subseteq N_2$ of $T$. The direct limit of flat $B_0$-modules over a directed index set is a flat $B$-module \cite[Prop.~5.34]{Rotman}, hence 
	$$
	\Omega \hbox{ is a flat }B_0\hbox{-module}.
	$$
In order to prove  that $\K$ is a flat  left $B_0$-module we need  to show that \begin{equation} \label{final2}  Tor_1^{B_0}( N , \K) = 0  \end{equation}
for any right $B_0$-module $N.$  Note that $N$ is the direct limit  (in this case the union)  of its finitely generated $B_0$-submodules  and the functor $Tor_i^{B_0}( ~, \K)$ commutes with direct limits for every $i \geq 1$ since tensor product commutes with direct limits (see \cite[Prop. 7.8]{Rotman}). Hence it suffices to prove (\ref{final2}) with  $N$ being a finitely generated $B_0$-module.

Consider the short exact sequence $$0 \to \Omega \to B_0 \to \K \to 0$$ of $B_0$-modules,
where the map $B_0 \to \K$ is given by $e_g\mapsto 1, g\in G.$  The corresponding long exact sequence in homology for $N$ is
	\begin{equation*} \label{long-exact}
	\ldots  \to  Tor_{1}^{B_0}( N , \Omega) \to   Tor_{1}^{B_0}( N , {B_0})  \to Tor_1^{B_0}( N , \K) \to
	N \otimes_{B_0} \Omega \to N \otimes_{B_0} B_0 \to N \otimes_{B_0} \K \to 0.
	\end{equation*}  
	 Note that $B_0$ is a free $B_0$-module, hence it is a projective $B_0$-module and $Tor_1^{B_0}( N , B_0) = 0.$  Consequently, we obtain the exact sequence
	\begin{equation*} \label{long-exact2}  0 \to  Tor_1^{B_0}( N , \K) \to
	N \otimes_{B_0} \Omega \to N \otimes_{B_0} B_0 \to N \otimes_{B_0} \K \to 0,
	\end{equation*} and thus (\ref{final2}) is equivalent to 
	\begin{equation} \label{final} N \otimes_{B_0} \Omega \to N \otimes_{B_0} B_0 \simeq N \hbox{ is an injective map} \end{equation}  for  any finitely generated $B_0$-module $N$.

	\begin{claim} \label{claim.induction.(6.9)}
 If (\ref{final}) holds for $B_0$-modules $N_1$ and $N_2$ and $0 \to N_1 \to N \to N_2 \to 0$ is a short exact sequence of $B_0$-modules, then (\ref{final}) holds for $N$.
	\end{claim}
	
	\medskip
 Indeed, the proof is a diagram chasing similar to the proof of  
	 the 3x3-Lemma (see \cite[Exer.~2.32]{Rotman}). Consider the commutative diagram
	 \begin{equation*}
	 \begin{matrix}
	 N_1 \otimes_{B_0} \Omega  & \mapnew{\alpha_1} & N_1 \otimes_{B_0} B_0 & \mapnew{} &N_1 \otimes_{B_0} \K & \mapnew{} & 0 \\
	 \hspace{9pt} \downarrow{\beta} & & \hspace{9pt}  \downarrow{\delta}&&\downarrow&& \\
	 N \otimes_{B_0} \Omega  & \mapnew{\alpha} & N \otimes_{B_0} B_0 & \mapnew{} &N \otimes_{B_0} \K & \mapnew{} & 0 \\
	 \hspace{9pt} \downarrow{\gamma} & & \downarrow&&\downarrow&& \\
	 N_2 \otimes_{B_0} \Omega  & \mapnew{\alpha_2} & N_2 \otimes_{B_0} B_0 & \mapnew{} &N_2 \otimes_{B_0} \K & \mapnew{} & 0 \\
	 \downarrow & &\downarrow&&\downarrow&& \\
	 0 & &0&&0&& \\
	 \end{matrix}
	 \end{equation*}
Observe that  $Im (\beta) = \Ker (\gamma),$ as the $B_0$-module  $\Omega $ is flat.  Then since $\alpha_1, \alpha_2$ and $\delta$ are injective, a simple diagram chasing implies that $\alpha$ is injective. This completes the proof of the Claim.
 
 \medskip
Using induction on the number of generators, Claim \ref{claim.induction.(6.9)} implies that (\ref{final}) holds for every finitely generated $B_0$-module precisely when  (\ref{final}) holds 
 for  every  cyclic $B_0$-module $N = B_0/I$. Consider the short exact sequence $0 \to I \to B_0 \to B_0/ I \to 0$ of $B_0$-modules and the corresponding long exact sequence in homology
\begin{equation*} 
	\ldots \to 0 = Tor_1^{B_0}( B_0 , \K) \to  Tor_1^{B_0}( {B_0}/I , \K) \to  I \otimes_{B_0} \K \to \end{equation*} \begin{equation}  \label{exact-f} {B_0} \otimes_{B_0} \K \Vightarrow{\partial} ({B_0}/I) \otimes_{B_0} \K   \to 0,
	\end{equation}
	 where we have used that $B_0$ is a free $B_0$-module, hence it is a  projective $B_0$-module and thus $ Tor_1^{B_0}( B_0 , \K) = 0$. Since $\partial$ is an epimorphism, and ${B_0} \otimes_{B_0} \K$ is a $\K$-vector space of dimension $1,$ we have the following two  cases.
	
	a) If $(B_0/I) \otimes_{B_0} \K  \simeq \K$ then  since $ {B_0} \otimes_{B_0} \K \simeq \K$ and $\partial$ is an epimorphism we can deduce that $\partial$ is an isomorphism.  
	Furthermore, for every right $B_0$-module $M$ we have $M \otimes_{B_0} \K \simeq M \otimes_{B_0} (B_0/\Omega) \simeq M / M \Omega $, hence 
	$$\K \simeq (B_0/I) \otimes_{B_0} \K \simeq  (B_0/I) \otimes_{B_0} (B_0/\Omega) \simeq B_0/ (I + \Omega),$$ which implies  $B_0 \not= I + \Omega$ that, together with the fact that $\Omega$ is a maximal ideal (i.e. $B_0$-submodule of $B_0$ with $B_0 / \Omega \simeq \K$ ),     implies  $I \subseteq \Omega$. Note that the long exact sequence (\ref{exact-f}) together with the fact that $\partial$ is injective imply $$Tor_1^{B_0}( {B_0}/I , \K) \simeq  I \otimes_{B_0} \K \simeq I \otimes_{B_0} B_0 / \Omega \simeq I/ I \Omega.$$ This last quotient is the zero module; in fact, if $b \in I \subseteq \Omega$, then $b \in \sum_{1 \leq i \leq k} ( 1 - e_{g_i})B_0$ for some $g_1, \ldots, g_k \in T$. Hence
	$b e_{g_1} \ldots e_{g_k} = 0$ and $b = b ( 1 - e_{g_1} \ldots e_{g_k} ) \in b \Omega \subseteq I \Omega$, that is,  $I \subseteq I \Omega$ and so $$Tor_1^{B_0}( {B_0}/I , \K) \simeq I / I \Omega = 0.$$

	b)
	If $({B_0}/I) \otimes_{B_0} \K  = 0$ then,  since, as above,  $({B_0}/I) \otimes_{B_0} \K \simeq B_0 / (I + \Omega)$, we deduce that $I \not\subseteq \Omega$ and by the long exact sequence (\ref{exact-f}) there is a short exact sequence $$0 \to Tor_1^{B_0}( {B_0}/I , \K)  \to I \otimes_{B_0} \K \to \K \to 0.$$ Since $I \otimes_{B_0} \K \simeq I / I \Omega$ we obtain
	\begin{equation} \label{final-t2} dim_{\K}(I / I \Omega) = dim_{\K}(I \otimes_{B_0} \K) = 1 + \dim_{\K} Tor_1^{B_0}( {B_0}/I , \K). \end{equation}
	On the other hand,  
 $I / (I \cap \Omega) \simeq 
	(I + \Omega ) /  \Omega = {B_0}/ \Omega \simeq \K$, hence
	$$\dim_{\K} I / (I \cap \Omega) = 1.$$ 
Notice that $I \cap \Omega = I  \Omega .$ 	Indeed,
as in a) an element $b \in I \cap \Omega $ can be written as  
$b = b ( 1 - e_{g_1} \ldots e_{g_k} )\in I  \Omega ,$ so that 
$I \cap \Omega \subseteq I  \Omega ,$ and thus $I \cap \Omega = I  \Omega .$ It follows that $\dim_{\K} I / I \Omega $ =1 and 
by (\ref{final-t2}) $$Tor_1^{B_0}( {B_0}/I , \K) = 0.$$ 
\end{proof}
\begin{theorem} \label{final-flat}
	Suppose that $\Delta$ is a connected component of the groupoid $\Gamma(G)$ with finitely many vertices. Then $\K \Delta$ is flat as a left $\kpg$-module.
\end{theorem}

\begin{proof} As the proof is long we split it in several steps.
	
	\medskip
	{\it Step 1.} In this step we prove  that
	\begin{equation} \label{kernel-1}
	\Ker (\lambda_{\Delta}) = \kpg (\Ker(\lambda_{\Delta}) \cap B) = (\Ker(\lambda_{\Delta}) \cap B) \kpg .
	\end{equation} 
	
	\medskip
	Recall that
	$$
	\lambda_{\Delta}([g]) = \sum_{g^{-1} \in C, C \in \V_{\Delta}} (C,g)
	$$
	and that 
	$\{ (C,1) \}_{ C \in \V_{\Delta} }$ is a basis of 
	$\lambda_{\Delta}(B)$ as $\K$-vector space (by Lemma \ref{Bdelta}). 
	Thus $$\lambda_{\Delta}(B [g]) = V_g,$$ where $V_g$ is the $\K$-subspace of $\K \Delta$ spanned by 
	$\{ (C,g) \}_{C \ni g^{-1}, C \in \V_{\Delta} }$. Note that $$V_g \cap  \sum_{t \in G,\\t\neq g} V_t = 0.$$
	
	Suppose $m = \sum_{g \in G} b_g [g] \in \Ker (\lambda_{\Delta})$, where $b_g \in B$. Then $$0 = \sum_{g \in G} \lambda_{\Delta} (b_g [g]) \in \oplus_{g \in G} V_g,$$ hence $$ \lambda_{\Delta}(b_g) \lambda_{\Delta}([g]) = \lambda_{\Delta}(b_g[g]) = 0 \hbox{ for every }g \in G.$$	
	Then  if $$\lambda_{\Delta}(b_g) = \sum_i k_i (C_i, 1) $$ for some $k_i \in \K \setminus \{ 0 \}$, we have $(C_i, 1) (C,g) = 0$ for every $ C \in \V_{\Delta}$ such that $g^{-1} \in C $, hence $C_i \not= g C$ i.e. $g \not\in C_i$. Recall that $$\lambda_{\Delta}(e_g) = \sum_{g \in C, C \in \V_{\Delta}} (C,1) \ \   \hbox{  and } \  \ 
	\lambda_{\Delta}(1 - e_g) = \sum_{g \not\in C,  C \in \V_{\Delta}} (C,1).$$
	Thus $\lambda_{\Delta}(b_g) \in \lambda_{\Delta} (B(1 - e_g))$, hence
	$$b_g \in B( 1 - e_g) + (B \cap \Ker (\lambda_{\Delta}) )$$ and $$ b_g [g] \in B ( 1 - e_g) [g]  +  (B \cap \Ker (\lambda_{\Delta}) ) [g] = (B \cap \Ker (\lambda_{\Delta}) ) [g], $$
	where we used that $( 1 - e_g) [g] = 0$. Note that
	$$ (B \cap \Ker (\lambda_{\Delta}) ) [g] = [g] (B \cap \Ker (\lambda_{\Delta}) )  \ \hbox{ for } g \in G,$$
	hence (\ref{kernel-1}) holds.
	
\medskip
	{\it Step 2.} As a consequence of Step 1, we shall show that there is a natural isomorphism
	\begin{equation*}
	- \otimes_{\kpg} \Ker (\lambda_{\Delta}) \simeq - \otimes_B (\Ker(\lambda_{\Delta}) \cap B).
	\end{equation*}
	
	\medskip
	 For an arbitrary right $\kpg$-module $M$ we have the well-defined map 
	$$ M \otimes_B ( \Ker(\lambda_{\Delta}) \cap B) \to M \otimes_{\kpg} \Ker(\lambda_{\Delta}), \;\;\; 
	m\otimes _B b \mapsto m \otimes _{\kpg} b,$$ which is clearly natural in $M.$ For the inverse map note that  $\kpg = \oplus _{g\in G} [g]D_{g\m},$ where $D_{g\m}$ is the ideal of $B$ generated by the element $e_{g\m}= [g\m][g].$ It follows by Step 1 that  an  arbitrary element $x \in \Ker (\lambda_{\Delta}) $ can be uniquely  written in the form 
	$$x= \sum _{g\in G} [g]d_{g\m} \in \kpg , \;\;\; (d_{g\m}\in D_{g\m} \cap \Ker (\lambda_{\Delta}) ). $$ 
	Then taking  any $m\in M,$ set
	$$m\otimes _{\kpg} x \mapsto \sum _{g\in G} m[g] \otimes _B d_{g\m}.$$ This is a well-defined map
	$  M \otimes_{\kpg} \Ker(\lambda_{\Delta}) \to M \otimes_B ( \Ker(\lambda_{\Delta}) \cap B) ,$ because for any $h\in G,$ on the one hand  we have
\begin{align*}&(m[h],  [g]d_{g\m}) \mapsto m[h][g]\otimes _B d_{g\m} =
  m[h][g]e_{g\m}\otimes _B d_{g\m} =\\
 &
 m[hg]e_{g\m}\otimes _B d_{g\m} =  
 m[hg]\otimes _B e_{g\m} d_{g\m}=m[hg]\otimes _B  d_{g\m},
 \end{align*} and on the other,
$ [h][g]d_{g\m} =  [h][g]e_{g\m}d_{g\m}  = 
 [hg]e_{g\m} d_{g\m}=  [hg] d_{g\m}= [hg]e_{(hg)\m } d_{g\m}, $ so that
$$(m, [h][g]d_{g\m})   \mapsto
m[hg]\otimes _B e_{(hg)\m } d_{g\m} = m[hg]  e_{(hg)\m }\otimes _B d_{g\m} = m[hg]  \otimes _B d_{g\m}. $$ Clearly, the two maps are inverses to each other.

	\medskip
	{\it Step 3.} We aim to show that $\lambda_{\Delta}(B)$ is flat as $B$-module,  where we consider $\lambda_{\Delta}(B)$ as a left $B$-module via the restriction map $\lambda_{\Delta} |_B$.
	
	\medskip
	 Let $A$ be a vertex of $\Delta$, let $\V_{\Delta} = \{ g_1 A, \ldots, g_n A \}$ (with $g_1 = 1$) and let $H$ be the stabilizer of $A$. 
		Recall from  \eqref{isomorphism.kdelta.matrices} that there is an isomorphism of algebras
		\begin{equation*}
		\eta : \kd \to M_n (\K H ),  \ (g_iA , g) \mapsto  E_{ji}(g\m_j g g_i), 
		\end{equation*}
		where $g g_i A = g_j A$ (and hence $ g\m_j g g_i \in H $). Via this isomorphism  we identify $\lambda_{\Delta}(B)$ with the algebra of diagonal matrices with entries in $\K$. 
	Thus
	\begin{equation} \label{direct-sum-0} \lambda_{\Delta}(B) \simeq K_1 \oplus \ldots \oplus K_n,
	\end{equation} where each $K_i \simeq \K$ as a vector space but we need to specify the $B$-action on $K_i$.
For $C = g_i A$
	 the element $(C,1) \in \K\Delta$ is identified with $E_{ii}(1)$.
	Then $K_i = E_{ii}(\K)$ and $e_g \in B$ acts on $K_i$ via $\lambda_{\Delta}$,  i.e., $\lambda_{\Delta}(e_g)$ as an element of $M_n(\K H)$ is a diagonal matrix with diagonal $(k_1, \ldots, k_n) \in \K^n$ and $e_g$ acts on $K_i$  as multiplication by $k_i$. Note that $k_i = 1$ if $g \in g_i A$ and $k_i = 0$ if $g \not\in g_i A$.

	Let $B_i$ be the $\K$-subalgebra of $B$ generated by $\{ e_g \ | \ g \in g_i A \}$ and consider the epimorphism of $\K$-algebras
	$$\lambda_i : B \to B_i$$ that sends $e_g$ to $e_g$ for $g \in g_i A$ and   sends $e_g$ to 0 if $g \not\in g_i A$.
	Thus $B_i$ acts trivially on $K_i$ (i.e. each $e_g$ for $g \in g_i A$ acts as 1) and the $B$-action on $K_i$ is via the epimorphism $\lambda_i$.
	
	By (\ref{direct-sum-0}) $\lambda_{\Delta}(B)$ is flat as $B$-module if and only if $K_i$ is flat as $B$-module for all $1 \leq i \leq n$.
Thus it remains to show that $- \otimes_B K_i$ is an exact functor for $ 1 \leq i \leq n$. Note that
$$- \otimes_B K_i \simeq ( - \otimes_B B_i)	\otimes_{B_i} \K.$$
It suffices to show that both functors 
$ - \otimes_B B_i$ and $- 	\otimes_{B_i} \K$ are exact.

Note that for the inclusion $j_i : B_i \to B$ we have that $\lambda_i j_i = id_{B_i}$ but $j_i$ is not a homomorphism of $B$-modules, hence we cannot claim that $B_i$ is projective as $B$-module, where the $B$-action is  via $\lambda_i$.

Let $\widehat{B}_i$ be the  unital $\K$-subalgebra of $B$ generated by $\{ e_g ~ | ~ g \in G \setminus g_i A \}$. 
	Then
$$\xi :  B_i \otimes_{\K} \widehat{B}_i \to B , \ x \otimes y \mapsto xy$$
is an isomorphism of algebras and of right $\widehat{B}_i$-modules. 
Considering any right $B$-module also as a right 
$\widehat{B}_i$-module (by restricting the action of $B$ to the subalgebra $\widehat{B}_i$) we can show that 
\begin{equation}\label{eq:naturaliso}
- \otimes_{B} B_i \simeq - \otimes_{\widehat{B}_i} \K,
\end{equation}

where we view $\K$ as $\widehat{B}_i$-module with $e_g$ acting as 0 for $g \in G \setminus g_i A$. 
Indeed, $B \otimes_B B_i \simeq B_i$ and 
\[
B \otimes_{\widehat{B}_i} \K \simeq B_i \otimes_{\K} \widehat{B}_i \otimes_{\widehat{B}_i} \K \simeq B_i \otimes_{\K} \K \simeq B_i \simeq B \otimes  _B B_i. 
\]
Keeping in mind the involved actions on modules, one can directly check that the resulting isomorphism 
$B \otimes_{\widehat{B}_i} \K \to  B \otimes  _B B_i$ maps
 $b \otimes _{\widehat{B}_i} \alpha $ to $b \otimes _{B} \alpha,$
where $\alpha \in \K$.

Since both tensor functors are  right exact and commute with direct sums, using a free presentation of $M$ as a $B$-module,  we conclude that 
 $$
 M \otimes_{\widehat{B}_i} \K \simeq M \otimes_{B} B_i, \;\;\;
  m   \otimes_{\widehat{B}_i} \alpha \mapsto m \otimes_B \alpha , $$
for every right $B$-module $M.$ Clearly  this respects homomorphisms of right $B$-modules (i.e. it is natural in $M$), so that \eqref{eq:naturaliso} follows.

By Lemma \ref{lemma-flat} both functors  $- \otimes_{B_i} \K$ and $- \otimes_{\widehat{B}_i} \K$ are exact,  which completes the proof of  Step 3.

\medskip	
	
	{\it Step 4.} In this step we complete the proof by showing that  $Tor^{\kpg}_1(M, \K \Delta) = 0$ for an arbitrary right ${\kpg}$-module $M.$
	
	\medskip

 Consider the short exact sequence of left $\kpg$-modules $$0 \to \Ker (\lambda_{\Delta}) \to \kpg \  \Vightarrow{\lambda_{\Delta}} \  \K \Delta \to 0$$  and the associated long exact sequence in homology 
$$\ldots \to  
	 Tor^{\kpg}_{1}(M, \Ker(\lambda_{\Delta})) \to Tor^{\kpg}_1(M, \kpg) \to  Tor^{\kpg}_1(M, \K \Delta) \to$$ 
	 $$ M \otimes_{\kpg} \Ker(\lambda_{\Delta}) \to  M \otimes_{\kpg} \kpg \to  M \otimes_{\kpg} \K \Delta \to 0.$$
	 Since $ \kpg$ is a free $ \kpg$-module,  it is projective and, consequently,   $ Tor^{\kpg}_1(M, \kpg) = 0.$ 
Thus we obtain the exact sequence 
$$	 0 \to  Tor^{\kpg}_1(M, \K \Delta) \to M \otimes_{\kpg} \Ker(\lambda_{\Delta}) \to  M \otimes_{\kpg} \kpg \to  M \otimes_{\kpg} \K \Delta \to 0.$$
Hence the equality  $Tor^{\kpg}_1(M, \K \Delta) = 0$  is equivalent to the fact of 
	 $$ M \otimes_{\kpg} \Ker(\lambda_{\Delta}) \to  M \otimes_{\kpg} \kpg \simeq M $$
	 being an injective map.
	  By Step 2 we have
	 
	 \begin{equation}\label{eq:NatIso}
	  M \otimes_{\kpg} \Ker(\lambda_{\Delta}) \simeq M \otimes_B ( \Ker(\lambda_{\Delta}) \cap B).
	  \end{equation}
	 The short exact sequence of $B$-modules $$0 \to  \Ker(\lambda_{\Delta}) \cap B \to B \to \lambda_{\Delta}(B) \to 0$$ gives rise to the associated long exact sequence in homology 
	 $$
	 \ldots \to 0 = Tor^{B}_1(M, B) \to  Tor^{B}_1(M, \lambda_{\Delta}(B)) \to$$ $$ M \otimes_{B} ( \Ker(\lambda_{\Delta}) \cap B)  \to  M \otimes_B B \to  M \otimes_{B} \lambda_{\Delta}(B) \to 0.$$
	 Keeping in mind the proof of Step 2, it is easy to see that the isomorphism \eqref{eq:NatIso} maps the kernel of
	$ M \otimes_{\kpg} \Ker(\lambda_{\Delta}) \to  M \otimes_{\kpg} \kpg$ onto the kernel of 
	$M \otimes_{B} ( \Ker(\lambda_{\Delta}) \cap B)  \to  M \otimes_B B.$  
	 Then
	 $$
	  Tor^{\kpg}_1(M, \K \Delta)  \simeq \Ker \big(  M \otimes_{\kpg} \Ker(\lambda_{\Delta}) \to  M \otimes_{\kpg} \kpg  \big) \simeq $$ $$ \Ker \big(  M \otimes_{B} ( \Ker(\lambda_{\Delta}) \cap B)  \to  M \otimes_B B \big) \simeq  Tor^{B}_1(M, \lambda_{\Delta}(B)) = 0,$$
	 where the last equality comes from the fact that by  Step 3 $\lambda_{\Delta}(B)$ is a flat $B$-module.
\end{proof}

\begin{cor} {\bf (Corollary C)}
	For any group $G$ we have  $cd^{par}_{\K}(G) \geq cd_{\K} (G)$. In particular, if	   ${\rm char}(\K) = p > 0$ and $G$ has $p$-torsion  then  	   $cd^{par}_{\K}(G) = cd_{\K} (G) = \infty$.
\end{cor}

\begin{proof} We give two proofs.

	1. Suppose  $cd^{par}_{\K}(G)  = m < \infty$.
	Let
	$$ {\mathcal P} : 0 \to P_m \to \ldots \to P_1 \to P_0 \to B \to 0$$
	be a projective resolution of the right $\kpg$-module $B$. By Theorem \ref{final-flat} applied for the connected component $\Delta$ of the groupoid $\Gamma(G)$ with unique vertex $G,$ i.e. $\K \Delta =  \K G$, we have that $ - \otimes_{\kpg} \K G$ is an exact functor. Note that $B \otimes_{\kpg} \K G \simeq   B \otimes_{\kpg} W \simeq \K $, where $W$ is defined before Theorem~\ref{decomposition}  where $G$ acts trivially on $\K$.
The latter isomorphism  is a particular case of Proposition \ref{trivial}.   
Then  ${\mathcal Q} = {\mathcal P} \otimes_{\kpg} \K G$ is an exact complex and since each $P_i$ is projective as a $\kpg$-module, we have that $P_i \otimes_{\kpg} \K G$ is a projective $\K G$-module. Thus ${\mathcal Q}$ is a projective resolution of
 of length $m$ of the trivial right $\K G$-module  $B \otimes_{\kpg} \K G \simeq \K$,   i.e. $cd_{\K}(G) \leq m$.
	
	2. Suppose
	\begin{align*}
n & = cd_{\K} (G) \\&  = max \{ n_0  ~ | ~ \hbox{ there is a }\K G-\hbox{module }U \hbox{ such that } H^{ n_0}(G, U) \not= 0\} \leq \infty .
	\end{align*}
	Consider the connected component $\Delta$ of the groupoid $\Gamma(G)$ with unique vertex $G$. Then by Theorem \ref{general2}  applied to  $W = \K H = \K G$,
	$$
	H_{par}^{ n_0}(G, U) \simeq H^{ n_0}(G, U) \not= 0,
	$$
	hence $cd_{\K}^{par}(G) \geq n_0$ and so $cd_{\K}^{par}(G) \geq n$.

Finally, if   ${\rm char}(\K) = p > 0$ and $G$ has $p$-torsion, let $P$ be a cyclic subgroup of order $p$. Then $cd_{\K}(G) \geq cd_{\K}(P)$. Note that there is a free resolution $  \ldots \to F_i \to F_{i-1} \to \ldots \to F_0 \to \K \to 0$ such that each $F_i = \K P$ and the differentials $\partial_i : F_i \to F_{i-1}$ for $i \geq 1$ alternate between the  two maps  $D$ and $N$ ,  defined by $D(\lambda) = \lambda ( x-1)$ and $N(\lambda) = \lambda \sum_{0 \leq i \leq p-1} x^i,$ where $x$ is a fixed generator of $P$. Then using this  resolution to compute the (usual) cohomology of $P$ with coefficients in the trivial $\K G$-module $\K ,$ we readily obtain that  $H^i(P, \K) = \K$ for all $i\geq 0,$ and hence $cd_{\K}(P) = \infty$. 

\end{proof}

\begin{cor}Let $G$ be an infinite group. Then $\K G$ is finitely generated (actually cyclic) but not finitely presented as a $\kpg$-module via $\lambda_{\Delta}$, where $\Delta$ is the connected component of the groupoid $\Gamma(G)$ with unique vertex $G$.	
	\end{cor}
\begin{proof} We give two proofs.
	
	1. The first proof is homological.
	By Proposition \ref{not-projective}, $\K \Delta = \K G$ is not projective as a $\kpg$-module (via $\lambda_{\Delta}$) and by Theorem \ref{final-flat}, $\K \Delta$ is flat as a $\kpg$-module.
	By \cite[Thm.~3.56]{Rotman} a finitely presented flat module is projective, hence $\K G$ is not finitely presented as $\kpg$-module.
	
	2. The second proof is ring theoretic and uses the proof of  Theorem \ref{final-flat}. By Step 1 of the proof of  Theorem \ref{final-flat}, $\Ker(\lambda_{\Delta})$ is finitely generated as $\kpg$-module if and only if $\Ker(\lambda_{\Delta}) \cap B$ is finitely generated as $B$-module. This is equivalent to $\lambda_{\Delta}(B)$ being finitely presented as $B$-module (via $\lambda_{\Delta}$). If this is the case then by (\ref{direct-sum-0})   from the proof  of Step 3 of Theorem \ref{final-flat},  $K_1$ is finitely presented as $B$-module. Note that in our case, using the notation of Step 3 of Theorem \ref{final-flat}, $n = 1$, $A =G$ and $B = B_1$, hence $K_1 = \K$ is the trivial $B$-module i.e. each $e_g$ acts as 1. Then $\Omega$ the ideal of $B$ generated by by all elements $e_g - 1,$  $g \in G,$ is finitely generated as a $B$-module, which contradicts the fact of  $G$ being infinite.
\end{proof}

\section{Auxiliary facts on idempotents and cancellation}

In this section $R$ is a unital associative ring. 
\begin{rmk} \label{cor-cancel}
	 Suppose that $r_1 e_1 + \ldots + r_k e_k = 0$, where $e_1, \ldots, e_k$ are idempotents that commute with each other. Then for each $i =1,2,\ldots, k$ we have that 
	$$r_i \in R e_1 + \cdots Re_{i-1} + Re_{i+1}+ \cdots + R e_k + R ( 1 - e_i).$$
	In fact, 
	$$ r_i = r_i e_i + r_i (1 - e_i) =  \sum_{j \neq i } -  r_j e_j + r_i (1 - e_i), \ \  1 \leq i \leq k.  $$	
	These equalities may be arranged in a matrix form: if  $v_1 $ is the column vector  $(e_1, \ldots, e_k)^t$ and $b$ is the column vector $(r_1( 1 - e_1), \ldots, r_k ( 1 - e_k))^t$, then 
	\begin{equation}\label{equation.symmetric.matrix}
	(r_1, \ldots, r_k)^t = N v_1 + b,
	\end{equation}
	where $N = (n_{i,j})$ is the $k \times k$ matrix which has entries $n_{i,j} = - r_j$ if $i \neq j$ and $n_{i,i} = 0$. 
\end{rmk}  The next result shows that there is an equality analogous to \eqref{equation.symmetric.matrix} involving  a skew-symmetric matrix $M$. 
 It  will be needed in the next section for studying the left augmentation ideal $IG$ when $G = \Z$.

\bigskip  
\begin{lemma} \label{cancellationLemma}Suppose that $r_1 e_1 + \ldots + r_k e_k = 0$, where $e_1, \ldots, e_k$ are idempotents that commute with each other.    Let $v_1 $ be the column vector  $(e_1, \ldots, e_k)^t$. Then there are elements $b_1, \ldots, b_k \in R$ and a skew-symmetric matrix $M \in M_k(R)$ (i.e. $M^t = - M$ and all diagonal entries of $M$ are 0) such that 
	$$
	(r_1, \ldots, r_k)^t = M v_1 + b,
	$$
	where $b$ is the matrix column $(b_1( 1 - e_1), \ldots, b_k ( 1 - e_k))^t$ for some $b_1, \ldots, b_k \in R$.
\end{lemma}

\begin{proof}
	1. Suppose first that $k = 2$, i.e., $r_1 e_1 = - r_2 e_2$.
	 Then 
	 $r_1 = r_1 e_1 + r_1 ( 1 - e_1) = - r_2 e_2 + r_1 ( 1 - e_1)$
	 and we have 
	$$ r_1 = - r  e_2 + b_1 ( 1 - e_1)$$
	with $r=r_2$ and $b_1= r_1.$
	Since $(r_2 - r e_1) e_2  =  r_2 e_2 - r_2 e_2 e_1 = 0$ there is $b_2 \in R$ such that
	$$ r_2 = r e_1 +  b_2 ( 1 - e_2).$$
	Then we can set \[
	M=
	\left[ {\begin{array}{cc}
		0 & -r \\
		r & 0 \\
		\end{array} } \right] .
	\]

	2. The general case is obtained by induction on $k$.	
	Suppose the result holds for $k-1$. Note that $r_1 e_1 ( 1 - e_k) + \ldots + r_{k-1}  e_{k-1} ( 1 - e_k) = (r_1 e_1 + \ldots + r_{k-1} e_{k-1}) ( 1 - e_k) = - r_k e_k ( 1 - e_k) = 0$ and that $e_1( 1 - e_k), \ldots, e_{k-1} ( 1 - e_k)$ are idempotents that commute with each other. Then by the inductive hypothesis there is a skew-symmetric matrix $M_0 \in M_{k-1}(R)$ and $\widetilde{b}_1, \ldots, \widetilde{b}_{k-1} \in R$ such that
	$$
	(r_1, \ldots, r_{k-1})^t =$$ $$ M_0 (e_1(1 - e_k), \ldots, e_{k-1} (1 - e_k))^t + 
	(\widetilde{b}_1( 1 - e_1(1 - e_k)), \ldots, \widetilde{b}_{k-1} ( 1 - e_{k-1}(1 - e_k)))^t.
	$$
	Then
	$$
	r_k e_k = - (r_1, \ldots, r_{k-1}) (e_1, \ldots , e_{k-1})^t = $$ $$ - 
	(M_0 (e_1(1 - e_k), \ldots, e_{k-1} (1 - e_k))^t )^t (e_1, \ldots , e_{k-1})^t $$ $$ -(\widetilde{b}_1( 1 - e_1(1 - e_k)), \ldots, \widetilde{b}_{k-1} ( 1 - e_{k-1}(1 - e_k)))  (e_1, \ldots , e_{k-1})^t.
	$$
	Then by multiplying the above equality on the right with $e_k$ 
	we deduce that 
	$$
	r_k e_k = r_k e_k^2 = - 
	(M_0 (e_1(1 - e_k), \ldots, e_{k-1} (1 - e_k))^t )^t (e_1, \ldots , e_{k-1})^t e_k $$ $$ -
	(\widetilde{b}_1( 1 - e_1(1 - e_k)), \ldots, \widetilde{b}_{k-1} ( 1 - e_{k-1}(1 - e_k)))  (e_1, \ldots , e_{k-1})^t e_k = $$ $$ - (\widetilde{b}_1( 1 - e_1(1 - e_k)), \ldots, \widetilde{b}_{k-1} ( 1 - e_{k-1}(1 - e_k)))  (e_1, \ldots , e_{k-1})^t e_k = $$ $$  - (\widetilde{b}_1 e_k, \ldots, \widetilde{b}_{k-1}    e_k ) (e_1, \ldots , e_{k-1})^t  = - \widetilde{b}_1e_1e_k - \ldots - \widetilde{b}_{k-1}   e_{k-1}e_k ,$$
	hence
	$
	(r_k + \widetilde{b}_1 e_1 + \ldots + \widetilde{b}_{k-1} e_{k-1}) e_k = 0
	$ and so
	$
	r_k + \widetilde{b}_1 e_1 + \ldots + \widetilde{b}_{k-1} e_{k-1} \in R( 1 - e_k)$. Then there is $b \in R$ such that 
	$$
	r_k =  - \widetilde{b}_1 e_1 - \ldots - \widetilde{b}_{k-1} e_{k-1} + b ( 1 - e_k).
	$$
	Finally we can define $M \in M_k(R)$ in the following way: 
	
	\noindent
	the last row is $(- \widetilde{b}_1 e_1, \ldots , - \widetilde{b}_{k-1} e_{k-1}, 0)$, the last column is $(\widetilde{b}_1 e_1, \ldots , \widetilde{b}_{k-1} e_{k-1}, 0)^t$ and if we delete from $M$ the last row and   the last column we are left with $M_0 ( 1 - e_k)$.
	Then
	$$
	(r_1, \ldots, r_k)^t = M (e_1, \ldots, e_k)^t + (\widetilde{b}_1 ( 1 - e_1), \ldots, \widetilde{b}_{k-1}( 1 - e_{k-1}), b ( 1 - e_k))^t
	$$
	and it suffices to set $b_i = \widetilde{b}_i$ for $ 1 \leq i \leq k
	-1$ and $b_k = b$. 
\end{proof}

\section{The structure of the left augmentation ideal $IG$ for $G = \mathbb{Z}$} \label{section-Z}

\subsection{ Some remarks on the general case}
From now on $R = \kpg$.
Recall  that $R$ is a $\K$-vector 
space with a generating set $e_{g_1} e_{g_1 g_2} \ldots e_{g_1 \ldots g_n} \# g_1 \ldots g_n,$  $(n \geq 1),$ and that
$IG$ is the kernel of the augmentation map 
$$\varepsilon : R \to B
$$
which is given by $$\varepsilon(e_{g_1} e_{g_1 g_2} \ldots e_{g_1 \ldots g_n} \# g_1 \ldots g_n) = e_{g_1} e_{g_1 g_2} \ldots e_{g_1 \ldots g_n}.$$ 
Thus
$IG$ is generated as a $\K$-vector space by $\{ f_{g_1, \ldots, g_n} \}_{n \geq 1, g_1, \ldots, g_n \in G}$, where
$$f_{g_1,\ldots, g_n} = e_{g_1} e_{g_1 g_2} \ldots e_{g_1 \ldots g_n} \# g_1 \ldots g_n - e_{g_1} e_{g_1 g_2} \ldots e_{g_1 \ldots g_n} \# 1.
$$
We have the following product formula:
$$(b_1 \# g_1)(b_2 \# g_2) = b_1 [g_1] b_2 [g_1^{-1}] \# g_1 g_2$$
where $b_1 \in B e_{g_1}, b_2 \in B e_{g_2}$.
\begin{lemma}  $IG$ is a left $R$-module generated by $\{ f_g \}_{g \in G \setminus 1}$.
\end{lemma}
\begin{proof} It suffices to apply the relation below many times and use induction on $n$
	$$f_{g_1, \ldots, g_n} = (e_{g_1} \# g_1) f_{g_2, \ldots, g_n} + (e_{g_1 g_2} e_{g_1 g_2 g_3} \ldots e_{g_1 \ldots g_n} \# 1) f_{g_1}. $$
\end{proof}

Observe that 
\begin{equation} \label{eq1} 
 (e_{g_1} \# g_1) f_{g_2} + (e_{g_1 g_2} \# 1) f_{g_1} - (e_{g_1} \# 1) f_{g_1 g_2}  = 0   \ \ \hbox{ for all } g_1, g_2 \in G
\end{equation}
and
\begin{equation} \label{eq2}
 ((e_g -1) \# 1) f_g  = 0   \ \ \hbox{ for  all } g \in G.
\end{equation}

\begin{lemma}\label{general-IG} a)  There is an isomorphism of left $B$-modules $$IG = \oplus_{g \in G \setminus \{ 1 \} } B f_g \simeq \oplus_{g \in G \setminus \{ 1 \} } B e_g.$$  In particular, $IG$ is a projective left $B$-module;
	
	b) $IG$ is a left $R$-module via the multiplication in $R,$ moreover,   
	$$
	[g_1] f_{g_2} =  e_{g_1} f_{g_1 g_2} - e_{g_1 g_2} f_{g_1};
	$$ 
	
	c) for every $g \in G$  the set $$
	\widetilde{X}_g = \{ e_h f_{hg} - e_{hg} f_h \}_{h \in G} 
	$$
	generates $R f_g$ as a left $B$-module and  $R f_g = R f_{g^{-1}}$.
	
\end{lemma}

\begin{proof}
	a), b)  	Identifying $[g_1]$ with $e_{g_1} \# g_1$ and  $e_g$ with $e_g \# 1,$ we have by (\ref{eq1}) that
	\begin{equation} \label{rozden1}
	[g_1] f_{g_2} =  e_{g_1} f_{g_1 g_2} - e_{g_1 g_2} f_{g_1}.
	\end{equation} 	 
Hence $IG$ is a left $B$-module generated by $\{ f_g \}_{g \in G \setminus \{ 1 \}} $ and 
	$$IG = \sum_{g \in G \setminus \{ 1 \} } B f_g = \oplus_{g \in G \setminus \{ 1 \} } B f_g,
	$$
	and we have an isomorphism of left $B$-modules   $B f_g \simeq B / B(1 - e_g),$ observing that  $B(1-e_g)$ is the annihilator of 
$f_g$ in $B.$ Thus $IG$ as a $B$-module is isomorphic to $$\oplus_{g \in G \setminus \{ 1 \} } B / B(1 - e_g) \simeq \oplus_{g \in G \setminus \{ 1 \} } B e_g,$$ which  is  a projective $B$-module, since $B = B( 1 - e_g) \oplus B e_g$.
	
	c) 
	By (\ref{rozden1})
	$$R f_g = B f_g + \sum_{h \in G \setminus \{ 1 \}} B (e_h f_{hg} - e_{hg} f_h) =$$ $$ B f_g + B f_{g^{-1}} + \sum_{g_1, g_2 \in G \setminus \{ 1 \}, g_2^{-1} g_1 = g  } B(e_{g_2} f_{g_1} - e_{g_1} f_{g_2}) =
	$$
	$$
	B f_{g^{-1}} + \sum_{h \in G \setminus \{ 1 \}} B( e_{hg^{-1}} f_h   -     e_h f_{hg^{-1}}  ) = R f_{g^{-1}},
	$$
	where $f_{1} = 0,$  $e_{1} = 1$ and the latter equality is obtained from the first one taking $g^{-1}$ instead of 
$g.$ Note that $f_g = e_{1} f_g - e_g f_{1}$ and so
	$$
	\widetilde{X}_g = \{ e_h f_{hg} - e_{hg} f_h \}_{h \in G} =  \{ f_g \} \cup \{ e_h f_{hg} - e_{hg} f_h \}_{h \in G \setminus \{ 1 \}}
	$$
	generates $R f_g$ as a left $B$-module. \end{proof}

 \subsection{The case $G = \mathbb{Z}$} 
 
 \ \ \ \ \ \ \ \ \ \ \ \ \ \ \ \ \ \ \ \ \ \ \ \ \ \ \ \ \ \ \ \ \ \ \

Let $G = \mathbb{Z} = \langle g \rangle$. Write $f_i $ for $f_{g^i } = e_{g^i} \# g^i - e_{g^i} \# 1 $ and $e_i $ for $e_{g^i}$.
Thus $e_0 = 1$, $f_0 = 0$ and the basic relations  (\ref{eq1}) and (\ref{eq2}) are
\begin{equation} \label{carnival1} 
e_i f_{i+j} = e_{i+j} f_i + (e_i \# g^{i}) f_{j},
\hbox{ 
	and }
(e_i - 1) f_i = 0  \hbox{ for } i,j \in \mathbb{Z}.  \end{equation}

\begin{lemma}\label{lem:Annihilator} For each $i \in \mathbb{Z},$ $i \neq 0,$ the annihilator of $f_i$ in $R$ is $R(1-e_i).$
\end{lemma}
\begin{proof} The inclusion $ ann_R(f_{i})\supseteq  R ( 1 - e_{i})$  obviously follows from \eqref{carnival1}. For the opposite inclusion take an arbitrary $\lambda \in ann_R(f_{i})$ and write $$\lambda = \sum _{j=-l}^m b_j [g^j] $$ with  
$b_j \in B$ and $m,l\geq 0.$
We apply induction on $l + m$. For $l = 0 = m$ we have $\lambda = b \in B$ and $0 = \lambda f_i = b ([g^i] - e_i)$ implies $
b [g^i] = 0 = b e_i$, hence $\lambda = b \in B ( 1- e_i) \subseteq R ( 1 - e_i)$, as required. 

Assume now that $l + m > 0$.
Then
 $$0= \sum _{j=-l}^m b_j [g^j] f_i= \sum _{j=-l}^m b_j [g^j] (e_i[g^i]-e_i)=  \sum _{j=-l}^m b_j [g^j] ([g^i]-e_i)=  \sum _{j=-l}^m b_j e_{j} [g^{i+j}] - \sum _{j=-l}^m b_j e_{i+j} [g^j].$$
 Without loss of generality we can assume that $i > 0$.  Since 
$R= \oplus _{t \in G} B [t],$  the above equality implies  that
 $b_m e_{m} [g^{i+m}] = 0$. Hence $b_m e_m \in \{ b \in B \   | \   b [g^{i+m}] = 0 \} = B ( 1 - e_{i+m})$, thus $b_m \in B ( 1 - e_{i+m})$. Then 
 $$b_m[g^m] \in B ( 1 - e_{i+m}) [g^m] = B [g^m] ( 1 - e_i) \subseteq R ( 1 - e_i)$$ and
 $$\lambda_0 = \lambda - b_m [g^m] =   \sum _{j=-l}^{m-1} b_j [g^j] \in ann_R (f_i) = \{ r \in R \ | \ r f_i = 0 \}.$$
 By induction $\lambda_0 \in R ( 1 - e_i)$, hence $\lambda = \lambda_0 +  b_m [g^m] \subseteq R ( 1 - e_i)$.

\end{proof}

Note that
$$IG = \oplus_{i \in \mathbb{Z} \setminus \{ 0 \} } B f_i \subseteq R = \oplus_{i \in \mathbb{Z}}  B (e_i \# g^i),
$$ where $B f_{i} \simeq B / B (e_{i} - 1)$ and $R f_i = R f_{-i}$. Define for $k \geq 1$
$$
V_k = R f_1 + \ldots + R f_k = R f_{-1} + \ldots + R f_{-k} \subset IG
$$
and hence, using \eqref{carnival1},  we have $$ R f_j = \sum_{i \in \mathbb{Z}} B (e_i \# g^i) f_j = 
B f_j + \sum_{i \in \mathbb{Z} \setminus \{ 0 \}} B (e_i f_{i+j} - e_{i+j} f_i).
$$
Note that $V_k$ has a generating set $X_k$ as a $B$-module, where 
$$
X_k = \{  \pm (e_i f_{i+j} - e_{i+j} f_i) \}_{1 \leq j \leq k, i \in \mathbb{Z} } \supset \{\pm  f_1, \ldots,  \pm f_k, \pm f_{-1}, \ldots , \pm f_{-k} \}.$$ 
 Obviously the above definition will work if we remove 
the formally unnecessary sign $\pm$ but we prefer to keep it 
for symmetry.

We have a natural order in $\{ f_i \}_{i \in \mathbb{Z}}$  : $f_{i_1} < f_{i_2}$ if $|i_1| < |i_2|$ and $f_{-i} < f_i$ for $i > 0$.
For a non-zero element $v = \sum_i b_i f_i$, where $b_i \in B$, we define the {\bf leading term} as $b_{i_0} f_{i_0}$ where $f_{i_0}$ is maximal with $b_{i_0} f_{i_0} \not= 0$ and we call $b_{i_0}$ a {\bf leading coefficient}. Note that any element of $b_{i_0} + B( 1 - e_{i_0})$ is a leading coefficient of $v$ since $B( 1 - e_{i_0})$ is the annihilator of $f_{i_0}$ in $B$.  We call $f_{i_0}$ the {\bf degree} of  $v$ with respect to $\{ f_i \}$ and write $f_{i_0} = deg(v)$. The {\bf support} 
$ supp\, (v)$ of the non-zero element $v$ is the set of all $f_i$ for which $b_i f_i \not= 0$.

\begin{lemma} \label{leading-go-back}
	If $v \in V_k \setminus \{ 0 \}$ has degree $f_s$ with $|s| \geq k+1$ then any leading coefficient of
	 $v$ with respect to $\{ f_i \}$ belongs to $I$, where 
	
	a)  $I = B e_{s-1} + \ldots + B e_{s-k}  + B (1 - e_s)$  if $s > 0$;
	
	b) $I = B e_{s+1} + \ldots + B e_{s+k}  + B (1 - e_s)$ if $s < 0$.
	
\end{lemma}

\begin{proof}

		\begin{claim} \label{Claim1.Lemma9.3} If $v = \sum_i b_i x_i$, where each $b_i \in B, x_i \in X_k$, $$f_{s_0} = max \{ f_j ~ | ~ f_j \in \cup_i supp(x_i) \}$$ and $f_{s_0} \not\in supp(v)$, i.e. $f_{s_0} > f_s$, then there is a decomposition $v = \sum_i \widetilde{b}_i \widetilde{x}_i$ with $\widetilde{b}_i \in B, \widetilde{x}_i \in X_k$ and  
		$$f_{s_1} = max \{ f_j ~ | ~ f_j \in \cup_i supp(\widetilde{x}_i) \} < f_{s_0}.$$ 
		\end{claim}
	
Applying Claim \ref{Claim1.Lemma9.3} several times we obtain

\begin{claim}\label{Claim2.Lemma9.3}
If $v \in V_k \setminus \{ 0 \}$ then we can write $v = \sum_i b_i x_i$, where each $b_i \in B, x_i \in X_k$ and $max \{ f_j ~ | ~ f_j \in \cup_i supp(x_i) \} = f_s = deg (v)$.
\end{claim} 

Note that Lemma \ref{leading-go-back} easily follows  from Claim \ref{Claim2.Lemma9.3}. Indeed, if $s>0,$ then the leading term of $v$ is 
of the form  $be_i f_{i+j} , b\in B,   i+j=s,$ $1\leq j \leq k ,$ hence $e_i \in \{ e_{s-k}, \ldots , e_{s-1}\}.$ Recalling that the leading coefficient is given modulo $B(1-e_s),$ we obtain a). Case b) is symmetric. 

\medskip
{\it Proof of Claim \ref{Claim1.Lemma9.3}.} We arrange the indexes of $x_i$ so that $f_{s_0} \in supp (x_i)$  for $1 \leq i \leq d$ and
 $f_{s_0} \not\in supp (x_i)$ for $i>d.$ Thus $x_j = e_{i_j} f_{s_0} - e_{s_0} f_{i_j}$ for $ 1 \leq j \leq d$ and $f_{s_0}$ is greater than the degree of $\widetilde{v},$ where  $\widetilde{v} = \sum_{d+1 \leq j} b_j x_j,$ and 
$$v = b_1(e_{i_1} f_{s_0} - e_{s_0} f_{i_1}) + \ldots + b_d(e_{i_d} f_{s_0} - e_{s_0} f_{i_d}) + \widetilde{v}.
$$
 Since $f_{s_0} \not\in supp(v)$ we have that
$\sum_{1 \leq j \leq d}  b_j e_{i_j} f_{s_0} = 0$, i.e., $\sum_{1 \leq j \leq d}  b_j e_{i_j} \in ann_B (f_{s_0}) = B ( 1 - e_{s_0})$, which means that  for some $\widehat{b} \in B$ we have
$$
b_1 e_{i_1} + \ldots + b_d e_{i_d} + \widehat{b} (1 - e_{s_0}) = 0.$$
Then by Lemma \ref{cancellationLemma} there exist  some skew-symmetric $(d+1) \times (d+1)$-matrix $M = (m_{j,t} )$  with coefficients in $B$   and 
 $\bar{b}_1, \ldots,  \bar{b}_{d+1} \in B$ such that
\begin{equation} \label{matrix-new}
(b_1, \ldots, b_d, \widehat{b})^t = M (e_{i_1}, \ldots, e_{i_d}, 1 - e_{s_0})^t + (\bar{b}_1 ( 1 - e_{i_1}), \ldots, \bar{b}_{d} ( 1 - e_{i_k}), \bar{b}_{d+1} e_{s_0})^t.
\end{equation} 
Note that
$$
v = (\sum_{1 \leq j \leq d} b_j e_{i_j}) f_{s_0} - e_{s_0} (\sum_{1 \leq j \leq d} b_j f_{i_j}) + \widetilde{v} = - e_{s_0} (\sum_{1 \leq j \leq d} b_j f_{i_j}) + \widetilde{v}.
$$
	Then it follows by  (\ref{matrix-new}), using     $(1 - e_{s_0}) e_{s_0} = 0$, $(1 - e_{i_j})f_{i_j} = 0,$  that 
	$$e_{s_0} \sum_{1 \leq j \leq d} b_j f_{i_j} = e_{s_0} \sum_{1 \leq j \leq d} \sum_{1 \leq t \leq d}  m_{j,t} e_{i_t} f_{i_j}.
$$	
Since $m_{j,j} = 0$ and $m_{j,t} = - m_{t,j}$ we deduce that
$$e_{s_0} \sum_{1 \leq j \leq d} b_j f_{i_j} = e_{s_0} \sum_{1 \leq j < t \leq d} m_{j,t} (e_{i_t} f_{i_j} - e_{i_j} f_{i_t})$$ and hence we have a new decomposition
\begin{equation*} v = - e_{s_0} (\sum_{1 \leq j \leq d} b_j f_{i_j}) + \widetilde{v} =  - e_{s_0} \sum_{1 \leq j < t \leq d} m_{j,t} (e_{i_t} f_{i_j} - e_{i_j} f_{i_t})  + \widetilde{v}
\end{equation*}
\begin{equation}  \label{new-dec} 
\in \sum_{1 \leq j < t \leq d} B (e_{i_t} f_{i_j} - e_{i_j} f_{i_t})  + \widetilde{v}. \end{equation} 
Note that since $T_k = \{ e_{i_1} f_{s_0} - e_{s_0} f_{i_1}, \ldots, e_{i_d} f_{s_0} - e_{s_0} f_{i_d}\} \subseteq X_k$ with $f_{s_0}$ maximal element among the supports of the elements of $T_k$ then either

1. $s_0 \geq k+1, \{ i_1, \ldots, i_d \} \subseteq \{ s_0 - k, \ldots, s_0 - 1\}$;

or

2. $s_0 \leq - k - 1, \{ i_1, \ldots, i_d \} \subseteq \{ s_0 + 1, \ldots, s_0 + k \}$.

In both cases $\{ e_{i_t} f_{i_j} - e_{i_j} f_{i_t} \}_{1 \leq j < t \leq d} \subseteq X_k$. This together with the new decomposition (\ref{new-dec}) completes the proof of Claim \ref{Claim1.Lemma9.3}.

\medskip

\end{proof}

\begin{prop} \label{quotient} $$V_{k+1} / V_k \simeq R/ (R ( 1 - e_{k+1}
	) + \sum_{1 \leq i \leq k} R e_i).$$
\end{prop}

\begin{proof}
	Note that $V_{k+1} / V_k$ is a cyclic  $R$-module, thus
	$$
	V_{k+1} / V_k \simeq R / J_k,
	$$
	where $J_k$ is the left ideal of $R$ defined by
	$$ 
	J_k : = \{ r \in R \mid \ r  f_{k+1} \in V_k \}.
	$$ 
	To prove  the proposition we will show that \begin{equation} \label{neu2} J_k = R ( 1 - e_{k+1}
	) + \sum_{1 \leq i \leq k} R e_i.\end{equation}
	The easy part is the inclusion $R ( 1 - e_{k+1}
	) + \sum_{1 \leq i \leq k} R e_i \subseteq J_k$ that is equivalent to $1 - e_{k+1}, e_1, \ldots, e_k \in J_k$. To see this, note by \eqref{eq2} that $( 1 - e_{k+1}) f_{k+1} = 0$ and using \eqref{rozden1}, for $ 1 \leq i \leq k,$ 
	\begin{equation} \label{star123}
	e_i f_{k+1} = [g^i] [g^{-i}] f_{k+1} = \end{equation} $$[g^i] ( e_{-i} f_{k+1 -i} - e_{k+1 -i} f_{-i}) \in R ( e_{-i} f_{k+1 -i} - e_{k+1 -i} f_{-i}) \subseteq \sum_{-k \leq j \leq k} R f_j = V_k.
	$$
	 Thus to complete the proof of the proposition it remains to be shown that $J_k \subseteq R ( 1 - e_{k+1}
	 ) + \sum_{1 \leq i \leq k} R e_i$.

	 Let $$\lambda \in J_k.$$ By the  the structure of $R = \kpg$  we have a decomposition
	\begin{equation} \label{lam0}
	\lambda = \sum_{i} b_i ( e_i \# g^i) \in R,
	\end{equation} 
	where $b_i \in B$. Hence, using again  \eqref{rozden1},
	\begin{equation} \label{rozden2}
	v = \lambda f_{k+1} = \sum_i b_i( (e_i \# g^i)  f_{k+1}) = \sum_i b_i ( e_i f_{k+1 + i} - e_{k+1 + i} f_i).
	\end{equation}
	
	If $v = 0$ then $\lambda \in ann_R(f_{k+1}) =  R ( 1 - e_{k+1})$  in view of Lemma~\ref{lem:Annihilator}.
	
	\medskip
	Suppose now that $v \not= 0$ and
	let $f_s$ be the degree of $v$ with respect to $\{ f_i \}$.

	\medskip
	Case 1.  Suppose $|s| \leq k$. Then   
	\begin{equation} \label{neu1} 
	v = \lambda f_{k+1} \in \sum_{-k \leq i \leq k} B f_i \subseteq V_k.
	\end{equation} 
	We decompose $$\lambda = \lambda_1 + \lambda_2 + \lambda_3,$$ where $$\lambda_1 = \sum_{i < -k} b_i ( e_i \# g^i), ~ \lambda_2 = \sum_{-k \leq i \leq -1}  b_i ( e_i \# g^i) ~ \hbox{  and } ~ \lambda_3 = \sum_{i \geq 0} b_i ( e_i \# g^i).$$
	Note that by (\ref{rozden2}) and $f_0 = 0$ we have  $$v_1 : = \lambda_1 f_{k+1} \in \sum_{j< 0} B f_j, v_2 : = \lambda_2 f_{k+1} \in \sum_{-k \leq j \leq k} B f_j $$ and $$v_3 : =  \lambda_3 f_{k+1} \in \sum_{j>  0} B f_j.$$ Then by (\ref{neu1}) and the fact that $
	v = v_1 + v_2 + v_3$
	we deduce that 
	$$v_1 \in \sum_{-k \leq j< 0} B f_j ~ \hbox{ and } \ v_3 \in \sum_{0 < j \leq k} B f_j,$$ i.e.
	$$
	v_1 = \lambda_1 f_{k+1} = \sum_{i < -k} b_i ( e_i f_{k+1 + i} - e_{k+1 + i} f_i) \in \sum_{-k \leq j< 0} B f_j
	$$ and so 
	$$ v_1 = \sum_{t < 0}  a_t   f_t$$
	where $a_t = - b_t e_{k+1+t} + b_{t-k-1} e_{t-k-1}$ for $t < -k$ and $a_t = b_{t-k-1} e_{t-k-1} $ for $ -k \leq t < 0$. Then 
	\begin{equation} \label{rozden3} 
	(- b_t e_{k+1+t} + b_{t-k-1} e_{t-k-1}) f_t = 0 \hbox{ for } t < -k. \end{equation} 	
	Hence
	$$- b_t e_{k+1+t} + b_{t-k-1} e_{t-k-1} \in B(1 - e_t)$$ and by multiplying with $e_t$ we get
	\begin{equation} \label{rozden4} 
	b_t  e_t e_{k+1+t} = b_{t-k-1} e_{t-k-1} e_t  \hbox{ for } t < -k.\end{equation} Since only finitely many $b_t$ can be non-zero we deduce from (\ref{rozden4}) that
	 \begin{equation} \label{rozden5}  b_t e_t e_{k+1+t}  = 0 \hbox{ for } t < -k.\end{equation}
	 Note that (\ref{rozden5}) is equivalent to
	\begin{equation} \label{rozden6} 
	b_t e_t \in B(1 - e_{k+1+t}) \hbox{ for } t<-k.\end{equation} 	
Therefore for $t < -k$
$$b_t e_t \# g^t = b_t e_t e_t \# g^t \in B(1 - e_{k+1+t}) e_t \# g^t = B 
(1 - e_{k+1+t}) [g^t] = B [g^t] (1 - e_{k+1}),$$
 hence \begin{equation} \label{lam1} \lambda_1 \in R ( 1 - e_{k+1}).\end{equation} 
	Similar calculations can be done for $\lambda_3$ to show that \begin{equation} \label{lam333} \lambda_3 \in R ( 1 - e_{k+1}). \end{equation}  
	For completeness we give the details. Recall that 
	\begin{equation} \label{neu3}
	v_3 = \lambda_3 f_{k+1} = \sum_{i \geq 0} b_i ( e_i f_{k+1 + i} - e_{k+1 + i} f_i) \in \sum_{0 < j \leq k} B f_j
	\end{equation} and so 
	$$ v_3 = \sum_{t > 0}  c_t   f_t,$$
	where $c_t = - b_t e_{k+1+t} + b_{t-k-1} e_{t-k-1}$ for $t >  k$ and $c_t = - b_t e_{k+1+t} $ for $  0< t \leq k$. Then by (\ref{neu3})
	\begin{equation} \label{rozden31} 
	(- b_t e_{k+1+t} + b_{t-k-1} e_{t-k-1}) f_t = 0 \hbox{ for } t > k. \end{equation} 	
	Hence
	$$- b_t e_{k+1+t} + b_{t-k-1} e_{t-k-1} \in B(1 - e_t)$$ and by multiplying with $e_t$ we get
	\begin{equation} \label{rozden41} 
	b_t  e_t e_{k+1+t} = b_{t-k-1} e_{t-k-1} e_t  \hbox{ for } t > k.\end{equation} Since only finitely many $b_t$ can be non-zero we deduce from (\ref{rozden41}) that
	\begin{equation} \label{rozden51}  b_t e_t e_{k+1+t}  = 0 \hbox{ for } t \geq 0.\end{equation}
	Note that (\ref{rozden51}) is equivalent to
	\begin{equation} \label{rozden61} 
	b_t e_t \in B(1 - e_{k+1+t}) \hbox{ for } t \geq 0.\end{equation} 	
	Then for $t \geq 0$
	$$b_t e_t \# g^t = b_t e_t e_t \# g^t \in B(1 - e_{k+1+t}) e_t \# g^t = B 
	(1 - e_{k+1+t}) [g^t] = B [g^t] (1 - e_{k+1}).$$  This implies \eqref{lam333}.

	Finally since $[g^i] e_{-i} = [g^i]$ we have that \begin{equation} \label{lam2} \lambda_2 = \sum_{-k \leq i \leq -1} b_i (e_i \# g^i) \in \sum_{-k \leq i \leq -1} B [g^i] = \sum_{-k \leq i \leq -1} B [g^i] e_{-i} \subseteq \sum_{1 \leq j \leq k} R e_j. \end{equation}
	Then by (\ref{lam1}), (\ref{lam333}) and (\ref{lam2}) we obtain that
	$$\lambda = \lambda_1 + \lambda_2 + \lambda_3 \in R(1 - e_{k+1}) + \sum_{1 \leq j \leq k} R e_j.$$

	Case 2. Suppose $|s| \geq k+1$,  recalling that $f_s$ is the degree of $v$. 
	We can assume that $s \geq k+1$, the case $s \leq - k-1$ is similar.
	
	Recall that $$\lambda = \sum_{i} b_i e_i \# g^{i} = \sum_{i} b_i [g^i].$$ Denote by $c$  the highest $i$ for which 
$b_i [g^i] \not= 0.$  Then by (\ref{rozden2}) there are two posibilities:
	
	1) $b_c e_c f_{k+1 +c} \not=0$ and  $f_{k+1+c}$ is the degree of $v = \lambda f_{k+1}$, i.e., $k+1+c = s$;
	
	2) $b_c e_c f_{k+1 +c} =0.$
	
	\medskip
	If 1) holds the leading term of $v$ is $b_{s-k-1}e_{s-k-1} f_s$ and   $b_{s-k-1}e_{s-k-1}$ is a leading coefficient of $v$, so $b_{s-k-1} e_{s-k-1}
	 \not= 0$. 	
	By Lemma \ref{leading-go-back} $$ e_{ s- k-1} b_{s-k-1} \in B e_{s-k} + \ldots + B e_{s-1} + B (1 - e_s).$$ 
	Then by Remark \ref{cor-cancel}
	$$b_{s-k-1} \in B e_{s-k} + \ldots + B e_{s-1} + B (1 - e_s) + B ( 1 - e_{s-k-1})
	$$
	and
	$$
	b_{s-k-1} (e_{s-k-1} \# g^{s-k-1}) \in$$ $$ (B e_{s-k} + \ldots + B e_{s-1} + B (1 - e_s) + B ( 1 - e_{s-k-1}))  (e_{s-k-1} \# g^{s-k-1})= 
	$$
	$$
	(B e_{s-k} + \ldots + B e_{s-1} + B (1 - e_s) )  (e_{s-k-1} \# g^{s-k-1})=$$ \begin{equation} \label{incl1}  (e_{s-k-1} \# g^{s-k-1}) (B e_1 + \ldots + B e_k + B( 1 - e_{k+1})).\end{equation}
	Then since $(1 - e_{k+1}) f_{k+1} = 0$ we have
	$$
	(e_{s-k-1} \# g^{s-k-1}) (B e_1 + \ldots + B e_k + B( 1 - e_{k+1})) f_{k+1} =$$ \begin{equation} \label{lam3}  (e_{s-k-1} \# g^{s-k-1}) (B e_1 + \ldots + B e_k) f_{k+1} \subseteq V_k,
	\end{equation} 
	since by (\ref{star123})
	\begin{equation} \label{lam4} 
	(B e_1 + \ldots + B e_k) f_{k+1} \subseteq V_k.
	\end{equation}
 Then by (\ref{incl1}) and (\ref{lam3}) we have
	\begin{equation} \label{incl2}  
	b_{s-k-1} (e_{s-k-1} \# g^{s-k-1}) f_{k+1} \in V_k.\end{equation}
	
	Finally for $\widetilde{\lambda} =
	\lambda - b_{s-k-1} e_{s-k-1} \# g^{s-k-1} = \lambda - b_c e_c \# g^c \in R$ we have by (\ref{incl2})
	$$
	\widetilde{\lambda} \cdot f_{k+1} =	v - b_{s-k-1} (e_{s-k-1} \# g^{s-k-1}) \cdot f_{k+1}   \in V_k
	$$ and by (\ref{lam0}) $supp_G(\widetilde{\lambda}) = supp_G(\lambda) \setminus \{ b_{s-k-1} e_{s-k-1} \# g^{s-k-1} \}$ 
	has strictly fewer elements than $supp_G(\lambda)$, where we define the support with respect to $G$ as 
	$$supp_G(\sum_i b_i e_i \# g^i) = \{ b_i e_i \# g^i ~ | ~ b_i e_i \# g^i  \not= 0 \}.$$ Note that $supp_G$ should not be confused with $supp$ used in the proof of 
	Lemma \ref{leading-go-back}. Then by induction on the number of elements in $supp_G(\lambda)$, where  $supp_G(0)$ has 0 elements, we have  that 
	$supp_G(\widetilde{\lambda})\subseteq R e_1 + \ldots + R e_k + R(1 - e_{k+1}) $, hence, in view of \eqref{incl1}, $$supp_G(\lambda) = supp_G(\widetilde{\lambda}) \cup \{ b_{s-k-1} e_{s-k-1} \# g^{s-k-1} \} \subseteq R e_1 + \ldots + R e_k + R(1 - e_{k+1})$$
	and so $\lambda \in R e_1 + \ldots + R e_k + R(1 - e_{k+1})$.
	
	\medskip
	If 2) holds then $b_c e_c \in ann_B(f_{k+1+c}) =  B( 1 - e_{k+1+c})$ and hence
	$$
	b_c [g^c] = b_c e_c [g^c] \in B( 1 - e_{k+1+c}) [g^c] = B [g^c] ( 1 - e_{k+1}) \subseteq R ( 1 - e_{k+1}).
	$$
	Thus for $\widetilde{\lambda} = \lambda - b_c e_c \# g^c$ we have
	$$\widetilde{\lambda} f_{k+1} = \lambda f_{k+1} - b_c e_c \# g^c f_{k+1} = v - b_c[g^c] f_{k+1}  \in v - R ( 1 - e_{k+1}) f_{k+1} = v - 0 = v \in V_k,$$
	and we can continue as in the last paragraph of 1).

It follows from Case 1 and Case 2 that 
	$$J_k \subseteq R(1 - e_{k+1}) + \sum_{1 \leq j \leq k} R e_j,$$ completing the proof.  
	\end{proof}

 \begin{lemma} Let $G \not= 1$ be a group and $R = \kpg$. Then 
		
		a) $IG$ is not a free left $R$-module,
		
		b)  $\kpg$ has the invariant basis number (IBN) property i.e. $(\kpg)^n \simeq (\kpg)^m$ as left $\kpg$-modules implies that $n = m$.
	\end{lemma}
	
	\begin{proof}	
		a) Consider $\K$ as a right $B$-module via the epimorphism of $\K$-algebras $B \to \K$ that sends $e_g$ to $0$ for every $g \in G \setminus \{ 1 \}$ and sends $e_{1} = 1$ to $1$. We view $R$ as a left $B$-module via the multiplication in $R$. Note that 
		$$\K \otimes_B
		\kpg = \K \otimes_B ( \oplus_{g \in G} B e_g [g]) \simeq \oplus_{g \in G} (\K \otimes_B B e_g [g]) = $$ $$  \oplus_{g \in G} (\K e_g \otimes_B B  [g]) =   \K e_{1} \otimes_B B  [1] = \K \otimes_B B = \K. $$
		Thus for every non-zero left free $R$-module $F$ we have
		$\K \otimes_B F \not= 0.$

		We will show  that $$\K \otimes_B IG = 0,$$
		which implies that $IG$ is not a free $R$-module. Recall that 
		by Lemma \ref{general-IG} $IG = \oplus_{g \in G \setminus \{ 1 \}} B f_g$. Then using that
		$(1 - e_g) f_g = 0$ we have 
		$$\K \otimes_B IG =\K \otimes_B (\oplus_{g \in G \setminus \{ 1 \}} B f_g) \simeq  \oplus_{g \in G \setminus \{ 1 \}} \K \otimes_B  B f_g = $$ $$   \oplus_{g \in G \setminus \{ 1 \}} \K \otimes_B  B e_g f_g =    \oplus_{g \in G \setminus \{ 1 \}} \K e_g \otimes_B  B  f_g = 0. $$

		b) Suppose that $(\kpg)^n \simeq (\kpg)^m$ as left $\kpg$-modules and recall that by part a) $\K \otimes_B
		\kpg  \simeq \K$. Then there is an isomorphism of $\K$-vector spaces $$\K^n \simeq \K \otimes_B (\kpg)^n \simeq \K \otimes_B (\kpg)^m \simeq \K^m,$$ and hence $n = m$. 
		
\end{proof}

\begin{theorem}
	Let $G = \mathbb{Z}$ and $R = \kpg$. Then $I G = \cup_{k \geq 1} V_k$ such that 
	each $V_k$ is a projective $R$-module and $$IG \simeq \oplus_{k \geq 1} {P}_k$$ where ${P}_k = V_{k} / V_{k-1}$ is a projective left $R$-module and $V_0 = 0$. 
	In particular $I G$ is a projective $R$-module, hence $cd_{\K}^{par}(\mathbb{Z}) = 1$.  Furthermore, $IG$ is not a finitely generated  $R$-module and it is not a free left $R$-module.
\end{theorem}
\begin{proof} Note that 
	$V_1 = R f_1 \simeq R / R (1 - e_1)$ and by  Proposition~\ref{quotient}, $V_{k+1} / V_k \simeq R / (R e_1 + \ldots + R e_k + R (1 - e_{k+1}))$. Then by Lemma \ref{idempotents},   $V_1$ and each $P_{k+1} = V_{k+1} / V_k$ is a projective $R$-module, and therefore the short exact sequence of $R$-modules
	$$
	0 \to V_k \to V_{k+1} \to P_{k+1} \to 0,
	$$
	where $V_0 = 0$, splits i.e $V_{k+1} \simeq V_k \oplus P_{k+1}$.
	Thus
	$$
	IG = \cup_{k \geq 1} V_k \simeq \oplus_{k \geq 1} {P}_k
	$$ is a projective $R$-module.
	
	 We note that if $IG$ is a finitely generated $R$-module, say by a set $T$, then for some $V_k$ we have that $T \subseteq V_k$. Thus  $IG = V_k$ and $V_{k+1} = V_k$, a contradiction with $V_{k+1}/ V_k \simeq R / (R e_1 + \ldots + R e_k + R (1 - e_{k+1})) \not= 0$.	 
	
\end{proof}


\section*{Acknowledgments}

 The authors thank the anonymous referee for the useful comments.
This work was supported by the S\~ao Paulo Research Foundation (FAPESP) [ 2020/16594-0 to M. D., 2018/23690-6 to D. H.  K.], and the National Council for Scientific and Technological Development (CNPq) [ 306583/2016-0 and 309469/2019-8 to M.  M. A., 312683/2021-9 to M.  D. and  305457/2021-7  to D.  H.  K. ].

{}
\end{document}